\documentclass[a4paper,reqno,11pt]{article}

\usepackage[english]{babel}
\usepackage{epsfig}
\usepackage{amsfonts}
\usepackage{amsmath}
\usepackage{amssymb}
\usepackage{amsthm}
\usepackage{mathrsfs}
\usepackage{color,amscd}
\usepackage{currfile}
\usepackage{euscript}
\usepackage{mathtools}
\usepackage{tikz}
\usetikzlibrary{arrows.meta}
\usepackage{tikz-cd}
\usepackage[all]{xy}

\let\oldcolor\color
\renewcommand{\color}[1]{\oldcolor{#1}}  

\usepackage{amsthm}

\newtheorem{theorem}{Theorem}[section]
\newtheorem{lemma}[theorem]{Lemma}

\newtheorem{corollary}[theorem]{Corollary}

\theoremstyle{definition}
\newtheorem{definition}[theorem]{Definition}
\newtheorem{remark}[theorem]{Remark}

\newcommand{\C}{\mbb{C}}
\newcommand{\I}{\mc{I}}

\newcommand{\N}{\mbb{N}}

\newcommand{\R}{\mbb{R}}
\newcommand{\OO}{\Omega}

\newcommand{\mbb}{\mathbb}
\newcommand{\mc}{\mathcal}

\newcommand{\mr}{\mathrm}

\newcommand{\mscr}{\mathscr}

\newcommand{\sq}{\mbb{S}}

\newcommand{\bs}{{\tiny $\blacksquare$}}

\newcommand{\ui}{i}
\newcommand{\pa}{\mc{P}}

\newcommand\HH{\mathbb H}

\newcommand{\debar}{\overline{\partial}}
\newcommand{\debarI}{\debar_I}
\newcommand{\debarIn}{\debar_I^{\,n}}
\newcommand{\thetabar}{{\overline{\vartheta}\,}}
\newcommand{\thetabarm}{\overline{\vartheta}_m}

\newcommand{\sbs}{\mr{sl}}

\title{\bf Slice-by-slice and global smoothness of \\slice regular and polyanalytic functions}

\author{
\textsc{Riccardo Ghiloni}}

\date{
{\small
\textit{Department of Mathematics, University of Trento, I-38123, Povo-Trento, Italy}} 
\\
{\small \texttt{riccardo.ghiloni@unitn.it}}
}

\begin{document}


\maketitle



\begin{abstract}
The concept of slice regular function over the real algebra $\HH$ of quaternions is a genera\-lization of the notion of holomorphic function of a complex variable. Let $\OO$ be an open subset of $\HH$, which intersects $\R$ and is invariant under rotations of $\HH$ around $\R$. A~function $f:\OO\to\HH$ is slice regular if it is of class $\mscr{C}^1$ and, for all complex planes $\C_I$ spanned by $1$ and a quaternionic imaginary unit~$I$, the restriction $f_I$ of $f$ to $\OO_I=\OO\cap\C_I$ satisfies the Cauchy-Riemann equations associated to $I$, i.e., $\debarI f_I=0$ on $\OO_I$, where $\debarI=\frac{1}{2}\big(\frac{\partial}{\partial\alpha}+I\frac{\partial}{\partial\beta}\big)$. Given any positive natural number $n$, a function $f:\OO\to\HH$ is called slice polyanalytic of order $n$ if it is of class $\mscr{C}^n$ and $\debarIn f_I=0$ on $\OO_I$ for all $I$. We define global slice polyanalytic functions of order $n$ as the functions $f:\OO\to\HH$, which admit a decomposition of the form $f(x)=\sum_{h=0}^{n-1}\overline{x}^hf_h(x)$ for some slice regular functions $f_0,\ldots,f_{n-1}$. Global slice polyanalytic functions of any order $n$ are slice polyanalytic of the same order $n$. The converse is not true: for each $n\geq2$, we give examples of slice polyanalytic functions of order $n$, which are not global.

The aim of this paper is to study the continuity and the differential regularity of slice regular and global slice polyanalytic functions viewed as solutions of the slice-by-slice diffe\-rential equations $\debarIn f_I=0$ on $\OO_I$ and as solutions of their global version $\thetabar^nf=0$ on $\OO\setminus\R$. Our quaternionic results extend to the monogenic case.

\vspace{1em}


\vspace{.5em}

\end{abstract}


\section{Introduction}

In 2006 Gentili and Struppa \cite{GeSt2006CR} extended to quaternions the notion of holomorphic function of a complex variable by means of a slice-by-slice procedure. The real algebra $\HH$ of quaternions decomposes into the union of the complex planes $\C_I$ spanned by $1$ and the quaternionic imaginary units $I$. Equip each `complex slice' $\C_I$ of $\HH$ with the complex coordinates $\alpha+I\beta$, where $\alpha,\beta\in\R$. Consider a non-empty open subset $\OO$ of $\HH$, which intersects $\R$ and is circular, i.e., invariant under the rotations of $\HH$ around $\R$. A function $f:\OO\to\HH$ is said to be \emph{slice regular} if it is of class~$\mscr{C}^1$ in the usual real sense and, for each quaternionic imaginary unit $I$, the restriction $f_I$ of $f$ to $\OO_I:=\OO\cap\C_I$ satisfies the Cauchy-Riemann equations associated to $\C_I$, i.e., $\debarI f_I=0$ on $\OO_I$, where $\debar_I:=\frac{1}{2}\big(\frac{\partial}{\partial\alpha}+I\frac{\partial}{\partial\beta}\big)$. This notion of regularity of $f$ has a hybrid nature. On one hand, it is global in the sense that $f$ is of class $\mscr{C}^1$ on the whole genuine $4$-dimensional open set $\OO$. On the other, it requires that $f$ is a solution of the slice-by-slice PDEs $\debarI f_I=0$ on each $2$-dimensional open subset $\OO_I$ of $\C_I$. A key result of Gentili and Struppa \cite{GeSt2007Adv} asserts that, if $x_0$ is a point of $\OO\cap\R$, then the slice regular function $f:\OO\to\HH$ admits a quaternionic series expansion $f(x)=\sum_{h\in\N}(x-x_0)^ha_h$ with $a_h\in\HH$, locally at $x_0$ in $\OO$. Here the point is that $f$ admits such an expansion slice-by-slice, i.e., $f_I(x)=\sum_{h\in\N}(x-x_0)^ha_h^I$, because each $f_I$ is holomorphic; however, the coefficients $a_h^I=\frac{1}{h!}\frac{\partial^hf_I}{\partial\alpha^h}(x_0)=\frac{1}{h!}\frac{\partial^hf}{\partial\alpha^h}(x_0)$ do not depend on~$I$. This argument continues to work if we replace the condition `$f$ is of class $\mscr{C}^1$ on $\OO$' with the weaker condition `$f$ is a slice-by-slice $\mscr{C}^1$ function', i.e., $f_I$ is of class $\mscr{C}^1$ on $\OO_I$ for all~$I$. As a consequence, if $f:\OO\to\HH$ is a slice-by-slice $\mscr{C}^1$ function such that $\debarI f_I=0$ on $\OO_I$ for all $I$, then it admits a quaternionic series expansion locally at each real point of $\OO$; in particular, it is real analytic (and hence of class $\mscr{C}^1$) on a circular open neighborhood $U$ of $\OO\cap\R$ in $\OO$. Thus, $f$ is slice regular on $U$, because it is of class $\mscr{C}^1$ on $U$ and $\debarI f_I=0$ on $U_I$ for all $I$ (see Theorem \ref{thm:GP2013-sbs} for details). This is an example of affirmative answer (at least on $U$) to the following smoothness problem, we study in this paper:
\begin{itemize}
\item {\it Is it possible to give equivalent definitions of slice regularity for $f:\OO\to\HH$, weakening the condition of global $\mscr{C}^1$ regularity of $f$ on $\OO$ and/or making use of global versions of the slice-by-slice PDEs `$\;\debarI f_I=0$ on $\OO_I$ for all $I$'?}
\end{itemize}

This problem has several facets and, at the moment, it is a bit intricate. In \cite{AIM2011}, it was introduced the concepts of slice function $f=\I(F):\OO\to\HH$ induced by a stem function $F:\OO_i\to\HH\otimes_\R\C$, and of the slice derivative $\frac{\partial f}{\partial x^c}:=\I\big(\frac{\partial F}{\partial\overline{z}}\big):\OO\to\HH$, provided $F$ is of class~$\mscr{C}^1$. If this is the case, i.e., $f=\I(F)$ is a slice function with $F$ of class $\mscr{C}^1$, then we know that $\frac{\partial f}{\partial x^c}(x)=\debarI f_I(x)$ for all $I$ and for all $x\in\OO_I$. A global differential operator $\thetabar$, defined on the space of $\mscr{C}^1$ functions from $\OO\setminus\R$ to $\HH$, appeared in \cite{global2013}. This operator has the following basic property: given any $g:\OO\setminus\R\to\HH$ of class $\mscr{C}^1$, it holds $\thetabar g(x)=\debarI(g|_{\OO_I\setminus\R})(x)$ for all $I$ and for all $x\in\OO_I\setminus\R$. Consequently, slice regular functions on $\OO$ can be interpreted as the solutions $g:\OO\setminus\R\to\HH$ of the PDEs $\thetabar g=0$, which admits an extension $f:\OO\to\HH$ of class~$\mscr{C}^1$ on the whole $\OO$. In \cite{CGS2013}, the authors introduced the global operator $G$ defined on the space of $\mscr{C}^1$ functions from the whole~$\OO$ to $\HH$. The operator $G$ is strongly related to $\thetabar$; indeed, $G=2|\mr{Im}(x)|^2\thetabar$ on $\OO\setminus\R$, where $|\mr{Im}(x)|$ is the Euclidean norm in $\HH=\R^4$ of the imaginary part $\mr{Im}(x)$ of $x$.  In the same paper \cite{CGS2013}, it is proved that the set of distributional solutions of PDEs $G(f)=0$ on $\OO$ strictly contains the set of slice regular functions on $\OO$. Recently, a generalization of the preceding smoothness problem has come into play. In \cite{ADS1}, the authors defined the notion of slice polyanalytic functions. Given a positive natural number~$n$, a function $f:\OO\to\HH$ is called \emph{slice polyanalytic of order $n$} if it is of class $\mscr{C}^n$ and $\debarIn f_I=0$ on $\OO_I$ for all $I$. We introduce the concept of \emph{global} slice polyanalytic function of positive order $n$. We say that a function $f:\OO\to\HH$ is global slice polyanalytic of order $n$ if $f(x)=\sum_{h=0}^{n-1}\overline{x}^hf_h(x)$ for some slice regular functions $f_0,\ldots,f_{n-1}$. By the very definitions, the set of slice polyanalytic functions of order~$1$, the set of global slice polyanalytic functions of order $1$ and the set of slice regular functions coincide. However, the set of global slice polyanalytic functions of any order $n\geq2$ is strictly contained in the set of slice polyanalytic functions of the same order $n$.  

We refer to the monograph \cite{GeStoSt2013} for a survey of quaternionic slice analysis, to the papers \cite{AlgebraSliceFunctions,singular} for a recent account of the theory on real alternative $^*$-algebras and to the recent paper \cite{GS-2020-2} for the theory of slice regular functions on not necessarily circular domains. A theory of slice regular functions in several variables was recently introduced in \cite{GhPe2020}. It is worth recalling that the theory of slice regular functions has significant applications in various areas of mathematics, as quaternionic functional calculus (see e.g. \cite{libroverde,GMP,GMP2017,MO}), twistor theory (see e.g. \cite{alta-1,alta-sarf,GSS2014}), operator semigroup theory (see e.g. \cite{CoSa2011-evolution,GhRe2016,GhRe2018}) and mechanism science (see \cite{GST2020}).  

The paper is organized as follows. In Section \ref{subsec:srf} we study the mentioned smoothness problem for slice regular functions. In Theorem \ref{thm:GP2013}, we give a detailed description of the equivalences existing between the various interpretations of slice regularity \cite{AIM2011,global2013,CGS2013} we mentioned above. In Theorem \ref{thm:GP2013-sbs} we improve Theorem \ref{thm:GP2013} by introducing new slice-by-slice conditions equivalent to the slice regularity. In Section \ref{subsec:poly}, we introduce the notion of global slice polyanalytic function of any positive order $n$. We give examples of slice polyanalytic functions of each order $n\geq2$, which are not global slice polyanalytic functions of order $n$, see Remark \ref{1.14}. These examples also prove that slice polyanalytic functions of any order $n\geq2$ do not satisfy neither the identity principle nor the representation formula; consequently, some results of \cite{ADS1} are correct only when $n=1$. Finally, we investigate the above smoothness problem for global slice polyanalytic functions. In Section \ref{subsec:monogenic} we extend to the monogenic case the quaternionic results of the preceding two sections. The proof of our results are postponed to Section \ref{sec:proofs}.


\section{The results}

\subsection{Smoothness of slice regular functions}\label{subsec:srf}

As we said, $\HH$ denotes the real algebra of quaternions. Let $\sq_\HH=\{I\in\HH:I^2=-1\}$ be the $2$-sphere of quaternionic imaginary units and, for each $I\in\sq_\HH$, let $\C_I\subset\HH$ be the copy of the complex plane spanned by $1$ and $I$, i.e., $\C_I=\{\alpha+I\beta\in\HH:\alpha,\beta\in\R\}$. The real algebra $\HH$ decomposes into the union of the `complex slices' $\C_I$, i.e., $\HH=\bigcup_{I\in\sq_\HH}\C_I$; moreover, $\C_I=\C_{-I}$ and $\C_I\cap\C_J=\R$ if $I\neq\pm J$.
Given a subset $E$ of $\C$, the circularization $\OO_E$ of $E$ in $\HH$ is defined by
\[
\OO_E:=\{\alpha+I\beta\in\HH:\alpha,\beta\in\R,\alpha+i\beta\in E, I\in\sq_\HH\}.
\]

A subset $S$ of $\HH$ is called circular, or axially symmetric, if $S=\OO_E$ for some $E\subset\C$. 

{\it Fix a non-empty connected circular open subset $\OO$ of $\HH$ such that $\OO\cap\R\neq\emptyset$ and denote by $D$ the subset of $\C$ such that $\OO=\OO_D$.} Note that $D$ is a non-empty connected open subset of $\C$ invariant under complex conjugation and $D\cap\R\neq\emptyset$.

Equip $\HH=\R^4$ with its natural Euclidean topology and structure of real analytic manifold.

Given any $I\in\sq_\HH$, define
\begin{equation}\label{def:OO_I}
\OO_I:=\OO\cap\C_I
\end{equation}
and the map $\phi_I:D\to\OO_I$ by
\begin{equation}\label{def:phi_I}
\text{$\phi_I(\alpha+i\beta):=\alpha+I\beta \qquad \forall\alpha+i\beta\in D$ with $\alpha,\beta\in\R$.}
\end{equation}

Each $\phi_I$ is a real analytic isomorphism. Let $g:\OO_I\to\HH$ be a function such that $g\circ\phi_I:D\to\HH$ is of class $\mscr{C}^1$ (in the usual real sense). Define $\debarI g:\OO_I\to\HH$ by
\begin{equation} \label{def:debarI}
\debarI g(x):=\frac{1}{2}\left(\left(\frac{\partial{}}{\partial\alpha}+I\frac{\partial{}}{\partial\beta}\right)(g\circ\phi_I)\right)(\phi_I^{-1}(x)).
\end{equation}
Given a function $f:\OO\to\HH$ and $I\in\sq_\HH$, we define the function  $f_I:\OO_I\to\HH$ as the restriction of $f$ to $\OO_I$, i.e.,
\begin{equation}\label{def:f_I}
f_I:=f|_{\OO_I}.
\end{equation}

Let $O$ be a non-empty open subset of $\HH$. For each $n\in\N^*:=\N\setminus\{0\}$, we denote by $\mscr{C}^n(O,\HH)$ the set of functions  of class $\mscr{C}^n$ from $O$ to $\HH$. We denote by $\mscr{C}^\omega(O,\HH)$ the set of real analytic functions from $O$ to $\HH$. We also define $\mscr{C}^0(O,\HH)$ as the set of continuous functions from $O$ to~$\HH$.

\begin{definition}[{\cite[Definition 1.1]{GeSt2006CR}}] \label{def:GS}
A function $f:\OO\to\HH$ is called \emph{slice regular} if $f\in\mscr{C}^1(\OO,\HH)$ and, for all $I\in\sq_\HH$, it holds
\[
\debarI f_I=0 \quad\text{on $\OO_I$.}
\]
The set of slice regular functions from $\OO$ to $\HH$ is denoted by $\mc{SR}(\OO,\HH)$. \bs
\end{definition}

\begin{remark} \label{rem:GS-indep-from-I}
Suppose for a moment that $\OO$ is a ball $B_\rho$ of $\HH$ centered at $0$ of positive radius $\rho\in(0,+\infty]$, where $B_{+\infty}=\HH$. As we have just sketched in the introduction, a key result of the Gentili-Struppa slice regular function theory asserts that a function $f:B_\rho\to\HH$ is slice regular if, and only if, $f$ admits a series expansion of the form $f(x)=\sum_{h\in\N}x^ha_h$ with $a_h\in\HH$, see \cite[Theorem 2.7]{GeSt2007Adv}. The `if' implication is easily proved differentiating the series term by term:
\[\textstyle
\debarI\big(\sum_{h\in\N}x^ha_h\big)_I=\sum_{h\in\N}\debarI(x_I^ha_h)=\sum_{h\in\N}(\debarI x_I^h)a_h=0,
\]
where $x_I$ is the inclusion map $B_{\rho,I}:=B_\rho\cap\C_I\hookrightarrow\HH$ and $x_I^h:=(x_I)^h$. The `only if' implication is based on a very interesting `\emph{independence from $I$}' argument. By definition, each restriction $f_I:B_{\rho,I}\to\HH$ is holomorphic w.r.t. the complex structure induced by the left multiplication by $I$. Thus, $f_I$ admits a unique series expansion $f_I(x_I)=\sum_{h\in\N}x_I^ha^I_h$ on $B_{\rho,I}$, where $a^I_h=\frac{1}{h!}(\partial_I^hf_I)(0)$. Here $\partial_I$ denotes the operator $\frac{1}{2}\big(\frac{\partial}{\partial\alpha}-I\frac{\partial}{\partial\beta}\big)$ and $\partial_I^h$ is its $h^{\mr{th}}$ power. Since $\partial_I^hf_I(0)$ equals the $h^{\mr{th}}$-derivative $\frac{\partial^hf}{\partial\alpha^h}(0)$ of $f$ at $0$ in the direction $1$, it turns out that the coefficients $a^I_h$ do not depend on $I\in\sq_\HH$, and we are done. \bs
\end{remark}

Let $F:D\to\HH\otimes_\R\C$ be a function. Write $F$ as follows: $F=F_1+\ui F_2$ with $F_1,F_2:D\to\HH$. The function $F$ is said to be a stem function if
\begin{equation}\label{eq:stem1}
\text{$F_1(\overline{z})=F_1(z)$ and $F_2(\overline{z})=-F_2(z)$ for all $z\in D$.}
\end{equation}
A function $f:\OO\to\HH$ is said to be a (left) slice function if there exists a stem function $F=F_1+\ui F_2:D\to\HH\otimes_\R\C$ such that
\begin{equation} \label{eq:slice}
f(\alpha+I\beta)=F_1(\alpha+i\beta)+IF_2(\alpha+i\beta)
\end{equation}
for all $\alpha,\beta\in\R$ with $\alpha+i\beta\in D$ and for all $I\in\sq_\HH$. In this case we say that $f$ is induced by $F$ and we write $f=\I(F)$. The even-odd properties \eqref{eq:stem1} of the stem function $F$ ensure that $f=\I(F)$ is well-defined. Moreover,  given any $I\in\sq_\HH$, it holds:
\begin{equation}\label{eq:stem}
F_1(z)=\frac{1}{2}\big(f(\phi_I(z))+f(\phi_I(\overline{z}))\big)\quad\text{and}\quad F_2(z)=-\frac{I}{2}\big(f(\phi_I(z))-f(\phi_I(\overline{z}))\big)
\end{equation}
for all $z\in D$. Consequently, each slice function is induced by a unique stem function. Suppose now that the stem function $F=F_1+iF_2$ is of class $\mscr{C}^1$, i.e., $F_1$ and $F_2$ are. In this case we can define the function $\frac{\partial F}{\partial\overline{z}}:D\to\HH\otimes_\R\C$ by
\begin{equation}\label{def:debarF}
\frac{\partial F}{\partial\overline{z}}:=\frac{1}{2}\left(\frac{\partial F}{\partial\alpha}+\ui\frac{\partial F}{\partial\beta}\right)=\frac{1}{2}\left(\left(\frac{\partial F_1}{\partial\alpha}-\frac{\partial F_2}{\partial\beta}\right)+i\left(\frac{\partial F_1}{\partial\beta}+\frac{\partial F_2}{\partial\alpha}\right)\right).
\end{equation}
It is immediate to see that $\frac{\partial F}{\partial\overline{z}}$ is again a stem function. As a consequence, if $f=\I(F)$, then we can define the slice function $\frac{\partial f}{\partial x^c}:\OO\to\HH$ by
\begin{equation}\label{def:slicedebar}
\frac{\partial f}{\partial x^c}:=\I\left(\frac{\partial F}{\partial\overline{z}}\right).
\end{equation}

We write $\frac{\partial f}{\partial x^c}$ also as $\big(\frac{\partial}{\partial x^c}\big)(f)$ or $\big(\frac{\partial}{\partial x^c}\big)f$. We denote by $\mc{S}^1(\OO,\HH)$ the set of slice functions from $\OO$ to $\HH$ induced by stem functions of class $\mscr{C}^1$.

Since $D$ is connected and intersects $\R$, and $\OO=\OO_D$ is circular, Definition \ref{def:GS} of slice regular function on $\OO$ coincides with the one given in \cite[Definition 8]{AIM2011}; that is, $f:\OO\to\HH$ is slice regular in the sense of Definition \ref{def:GS} if, and only if, $f\in\mc{S}^1(\OO,\HH)$ and $\frac{\partial f}{\partial x^c}=0$ on $\OO$. This is the statement of equivalence $(\mr{a})\Leftrightarrow(\mr{e})$ in Theorem \ref{thm:GP2013} below. We refer to \cite{AIM2011} for further details on slice and slice regular functions.

Write each quaternion $x\in\HH$ as $x=x_0+x_1i+x_2j+x_3k$ with $x_0,x_1,x_2,x_3\in\R$. Define $\mr{Re}(x):=x_0$ and $\mr{Im}(x):=x-\mr{Re}(x)=x_1i+x_2j+x_3k$. Denote by $\overline{x}=\mr{Re}(x)-\mr{Im}(x)$ the conjugation of $x$ in $\HH$ and by $|x|=\sqrt{x\overline{x}}$ the Euclidean norm of $x$ in $\HH=\R^4$. Set
\begin{equation}\label{def:OO_*}
\OO_*:=\OO\setminus\R=\OO_{D\setminus\R}.
\end{equation}

Given a function $f:\OO\to\HH$, we define the function $f_*:\OO_*\to\HH$ as the restriction of $f$ to $\OO_*$, i.e.,
\begin{equation}\label{def:f_*}
f_*:=f|_{\OO_*}.
\end{equation}

\begin{definition}[{\cite[p. 564]{global2013}}]
The differential operator $\thetabar:\mscr{C}^1(\OO_*,\HH)\to\mscr{C}^0(\OO_*,\HH)$ is defined~by
\[
\thetabar:=\frac{1}{2}\left(\frac{\partial}{\partial x_0}+\frac{\mr{Im}(x)}{|\mr{Im}(x)|^2}\left(\sum_{h=1}^3x_h\frac{\partial}{\partial x_h}\right)\right).
\]
More explicitly, if $g\in\mscr{C}^1(\OO_*,\HH)$ and $x=x_0+x_1i+x_2j+x_3k\in\OO_*$, then
\[
\thetabar(g)(x):=\frac{1}{2}\left(\frac{\partial g}{\partial x_0}(x)+\frac{x_1i+x_2j+x_3k}{x_1^2+x_2^2+x_3^2}\left(x_1\frac{\partial g}{\partial x_1}(x)+x_2\frac{\partial g}{\partial x_2}(x)+x_3\frac{\partial g}{\partial x_3}(x)\right)\right). \;\;\text{\bs}
\]
\end{definition}

We refer to the paper \cite{Pe2019} for recent results concerning the operator $2\thetabar$.

Following \cite[Definition 1.6]{CGS2013}, we define the differential operator $G:\mscr{C}^1(\OO,\HH)\to\mscr{C}^0(\OO,\HH)$ by
\begin{equation}\label{def:G}
G:=|\mr{Im}(x)|^2\frac{\partial}{\partial x_0}+\mr{Im}(x)\left(\sum_{h=1}^3x_h\frac{\partial}{\partial x_h}\right).
\end{equation}
Evidently, if $f\in\mscr{C}^1(\OO,\HH)$ and $x\in\OO_*$, then $G(f)(x)=2|\mr{Im}(x)|^2\,\thetabar(f_*)(x)$. Thus, thanks to the density of $\OO_*$ in $\OO$, we deduce:
\begin{equation}\label{eq:thetabarG}
\text{\it Given $f\in\mscr{C}^1(\OO,\HH)$, $f\in\ker(G)$ if, and only if, $f_*\in\ker(\thetabar)$.}
\end{equation}

The next theorem follows from \eqref{eq:thetabarG} and a careful reading of (the context and) the statements of \cite[Propositions 7(3) and 8]{AIM2011} and \cite[Theorems 2.2 and 2.4]{global2013}, see the proof in Section \ref{sec:proofs}.

\begin{theorem}\label{thm:GP2013}
Let $f:\OO\to\HH$ be a function. The following assertions are equivalent.
\begin{itemize}
 \item[$(\mr{a})$] $f\in\mc{SR}(\OO,\HH)$, i.e., $f$ is slice regular.
 \item[$(\mr{b})$] $f\in\mscr{C}^0(\OO,\HH)$, $f_*\in\mscr{C}^1(\OO_*,\HH)$ and $f_*\in\ker(\thetabar)$.
 \item[$(\mr{c})$] $f\in\mscr{C}^1(\OO,\HH)$ and $f_*\in\ker(\thetabar)$.
 \item[$(\mr{d})$] $f\in\mscr{C}^1(\OO,\HH)$ and $f\in\ker(G)$.
 \item[$(\mr{e})$] $f\in\mc{S}^1(\OO,\HH)$ and $\frac{\partial f}{\partial x^c}=0$ on $\OO$.
\end{itemize} 
\end{theorem}

\begin{corollary}\label{cor:GP2013}
$\mc{SR}(\OO,\HH)=\{f\in\mscr{C}^1(\OO,\HH):f_*\in\ker(\thetabar)\}=\ker(G)\subset\mscr{C}^\omega(\OO,\HH)$. 
\end{corollary}

In the monogenic case, some versions of the equivalences stated in Theorem \ref{thm:GP2013} are known, see \cite[Lemma 2.2]{CoSa2009}, \cite[Lemma 2.13]{CoSa2011} and \cite[Section 3]{CGS2013}.

\begin{remark}\label{rem:distribution}
Following \cite[Definition 1.6]{CGS2013}, we define $\mc{GS}(\OO)$ as the set of all distributional solutions $f$ of the differential equation $G(f)=0$ on $\OO$. By \cite[proof of Theorem 3.5]{CGS2013} and \cite{CS2014}, there exist distributional solutions in $\mc{GS}(\OO)$, which are not functions. Thus, the set $\mc{SR}(\OO,\HH)=\ker(G)$ is strictly contained in $\mc{GS}(\OO)$. Equivalence $(\mr{a})\Leftrightarrow(\mr{b})$ of Theorem \ref{thm:GP2013} and \eqref{eq:thetabarG} ensure that $\mc{SR}(\OO,\HH)$ is the set of functions $f$ in $\mscr{C}^0(\OO,\HH)$ such that $f_*$ belongs to $\mscr{C}^1(\OO_*,\HH)$ and $G(f)=0$ on $\OO$ in the sense of distributions, i.e., $f\in\mc{GS}(\OO)$. \bs
\end{remark}

Let us introduce the concept of slice-by-slice regularity for functions $f:\OO\to\HH$.

\begin{definition}
We say that a function $f:\OO\to\HH$ is \emph{slice-by-slice continuous} if, for all $I\in\sq_\HH$, the restriction $f_I:\OO_I\to\HH$ is continuous or, equivalently, if $f_I\circ\phi_I:D\to\HH$ is continuous. We denote by $\mscr{C}^0_\sbs(\OO,\HH)$ the set of slice-by-slice continuous functions from $\OO$ to $\HH$.

Let $n\in\N^*$. We say that $f:\OO\to\HH$ is a \emph{slice-by-slice $\mscr{C}^n$} function if, for all $I\in\sq_\HH$, the composition $f_I\circ\phi_I:D\to\HH$ is of class $\mscr{C}^n$. We denote by $\mscr{C}^n_\sbs(\OO,\HH)$ the set of slice-by-slice $\mscr{C}^n$ functions from $\OO$ to $\HH$. \bs 
\end{definition}

If $f:\OO\to\HH$ is a slice-by-slice $\mscr{C}^1$ function and $\debarI f_I=0$ on $\OO_I$ for all $I\in\sq_\HH$, it is natural to ask whether $f$ is slice regular or, equivalently, $f$ is of class $\mscr{C}^1$ on $\OO$ (and hence it satisfies Definition \ref{def:GS}). As we will see in Theorem \ref{thm:GP2013-sbs} below, the answer is affirmative. In Theorem \ref{thm:GP2013-sbs} we also improve implication $(\mr{b})\Rightarrow(\mr{a})$ of Theorem \ref{thm:GP2013} by replacing $\mscr{C}^0(\OO,\HH)$ with $\mscr{C}^0_\sbs(\OO,\HH)$. To corroborate this improvement, we give an example of a function $f\in\mscr{C}^0_\sbs(\OO,\HH)$ such that $f_*\in\mscr{C}^1(\OO_*,\HH)$ and $f\not\in\mscr{C}^0(\OO,\HH)$.

\begin{remark}\label{rem:c0-sbs0}
Up to a translation of $\OO$ in $\HH$ along a suitable real number, we can assume that $0\in\OO$. Define the function $f:\OO\to\HH$ by setting
\begin{equation}\label{eq:sbscont-noncont}
f(x):=
\left\{
 \begin{array}{ll}
x_1^2x_2(x_1^4+x_2^2+x_3^2)^{-1}  & \text{ if $\,x\in\OO\setminus\R=\OO_*$,}
\vspace{.3em}
\\
0  & \text{ if $\,x\in\OO\cap\R$,}
\end{array}
\right.
\end{equation}
where $x=x_0+x_1i+x_2j+x_3k\in\OO$ with $x_0,x_1,x_2,x_3\in\R$. Evidently, $f_*$ belongs to $\mscr{C}^1(\OO_*,\HH)$ (actually, $f\in\mscr{C}^n(\OO_*,\HH)$ for all $n\in\N^*$). Let $I=i_1i+i_2j+i_3k\in\sq_\HH$, where $i_1,i_2,i_3\in\R$. Given $\alpha+i\beta\in D$ with $\alpha,\beta\in\R$, it holds:
\[
f_I(\alpha+I\beta)=
\left\{
\begin{array}{ll}
\beta i_1^2i_2(\beta^2i_1^4+i_2^2+i_3^2)^{-1}  & \text{ if $\,\beta\neq0$,}
\vspace{.3em}
\\
0  & \text{ if $\,\beta=0$.}
\end{array}
\right.
\]
Note that $f_I$ is constantly equal to $0$ if $i_2=0$. If $i_2\neq0$, then $\beta^2i_1^4+i_2^2+i_3^2>0$ for all $\beta\in\R$ and $f_I$ is continuous. It follows that $f\in\mscr{C}^0_\sbs(\OO,\HH)$ (actually, $f\in\mscr{C}^n_\sbs(\OO,\HH)$ for all $n\in\N^*$). Define the sequence $\{q_h\}_{h\in\N^*}$ of quaternions by $q_h:=h^{-1}i+h^{-2}j$. Note that it converges to~$0$. On the other hand, for $h$ sufficiently large, say $h\geq H$ for some $H\in\N^*$, it holds: $q_h\in\OO$ and $f(q_h)=2^{-1}$. It follows that the sequence $\{f(q_h)\}_{h\geq H}$ converges to $2^{-1}\neq0=f(0)$. Hence $f\not\in\mscr{C}^0(\OO,\HH)$. \bs
\end{remark}

We are ready to state our next result.

\begin{theorem}\label{thm:GP2013-sbs}
Let $f:\OO\to\HH$ be a function. The following assertions are equivalent.
\begin{itemize}
 \item[$(\mr{a})$] $f\in\mc{SR}(\OO,\HH)$.
 \item[$(\mr{a}')$] $f\in\mscr{C}^1_\sbs(\OO,\HH)$ and $\debarI f_I=0$ on $\OO_I$ for all $I\in\sq_\HH$.
 \item[$(\mr{b}')$] $f\in\mscr{C}^0_\sbs(\OO,\HH)$, $f_*\in\mscr{C}^1(\OO_*,\HH)$ and $f_*\in\ker(\thetabar)$.
 \end{itemize}
\end{theorem}


\subsection{Smoothness of global slice polyanalytic functions of higher order}\label{subsec:poly}

Let $n\in\N^*$, let $I\in\sq_\HH$ and let $g:\OO_I\to\HH$ be a function such that $g\circ\phi_I:D\to\HH$ is of class~$\mscr{C}^n$. Define $\debarIn g:\OO_I\to\HH$ by
\begin{equation} \label{def:debarIn}
\debarIn g(x):=\frac{1}{2^n}\left(\left(\frac{\partial{}}{\partial\alpha}+I\frac{\partial{}}{\partial\beta}\right)^n(g\circ\phi_I)\right)(\phi_I^{-1}(x)).
\end{equation}
Evidently, one can define $\debarIn g$ also by induction: $\debarI^{\,0}g:=g$ and $\debarIn g:=\debarI(\debarI^{\,n-1}g)$ if $n\geq1$.

In \cite{ADS1} the authors introduce the concept of quaternionic slice polyanalytic function of order~$n$. Its definition is the natural generalization of Definition~\ref{def:GS} in which one replaces $\debarI$ with $\debarIn$.

\begin{definition}[{\cite[Definition 3.1]{ADS1}}]\label{def:sl-poly-sbs}
Let $n\in\N^*$. A function $f:\OO\to\HH$ is said to be \emph{slice polyanalytic of order $n$} if $f\in\mscr{C}^n(\OO,\HH)$ and, for each $I\in\sq_\HH$, it holds:
\[
\debarIn f_I=0 \quad\text{on $\OO_I$.}
\]
The set of slice polyanalytic functions of order $n$ from $\OO$ to $\HH$ is denoted by $\mc{SP}_n(\OO,\HH)$, i.e.,
\[
\text{$\mc{SP}_n(\OO,\HH):=\{f\in\mscr{C}^n(\OO,\HH):\debarIn f_I(x)=0\;\; \forall I\in\sq_\HH,\forall x\in\OO_I\}$. \bs}
\]
\end{definition}

In the theory of polyanalytic functions of a complex variable, the following is a basic result.

\begin{theorem}[{\cite[Section 1.1, (1.1) and (1.3)]{balk}}]\label{thm:classical}
Let $n\in\N^*$, let $E$ be a non-empty open subset of $\C$ and let $g:E\to\C$ be a function of class $\mscr{C}^n$. Then $\frac{\partial^ng}{\partial\overline{z}^n}=0$ on $E$ if, and only if, there exist holomorphic functions $g_0,\ldots,g_{n-1}:E\to\C$ such that $g(z)=\sum_{h=0}^{n-1}\overline{z}^hg_h(z)$ for all $z\in E$.

If $g$ satisfies the preceding equivalent conditions, then it is said to be \emph{polyanalytic of order $n$}.
\end{theorem}

The latter result suggests introducing the following variant of Definition \ref{def:sl-poly-sbs}.

\begin{definition}\label{def:sl-poly-global}
Let $n\in\N^*$ and let $f:\OO\to\HH$ be a function. We say that $f$ is a \emph{global slice polyanalytic function of order $n$} if there exist $f_0,\ldots,f_{n-1}\in\mc{SR}(\OO,\HH)$ such that
\begin{equation} \label{eq:decomposition}
f(x)=\sum_{h=0}^{n-1}\overline{x}^hf_h(x)\quad\text{for all $x\in\OO$.}
\end{equation}
We denote by $\mc{SP}^n(\OO,\HH)$ the set of global slice polyanalytic functions of order $n$ from $\OO$ to $\HH$. \bs
\end{definition}

By the very definitions, the concepts of slice regular function, slice polyanalytic function of order~$1$ and global slice polyanalytic function of order~$1$ coincide, i.e.,
\[
\mc{SR}(\OO,\HH)=\mc{SP}_1(\OO,\HH)=\mc{SP}^1(\OO,\HH).
\]

In \cite{ADS1} the authors prove a decomposition theorem for slice polyanalytic functions of any order $n\in\N^*$, which extends Theorem \ref{thm:classical} to the slice setting. Making use of this theorem, they deduce the \emph{identity principle}, the \emph{representation formula} and the \emph{sliceness condition} for slice polyanalytic functions of any order, see \cite[Theorem 3.8, Theorem 3.9, Corollary 3.11]{ADS1}, respectively. Bearing in mind our Definition \ref{def:sl-poly-global}, the mentioned decomposition theorem can be equivalently restated as follows.

\begin{theorem}[{\cite[Corollary 3.7]{ADS1}}]\label{thm:poly-decomposition}
For all $n\in\N^*$, it holds: $\mc{SP}_n(\OO,\HH)=\mc{SP}^n(\OO,\HH)$.
\end{theorem}

The inclusion $\mc{SP}^n(\OO,\HH)\subset\mc{SP}_n(\OO,\HH)$ follows easily from Leibniz's rule for $\debarI$ and the slice re\-gularity of the $f_h$'s. Let $f\in\mc{SP}^n(\OO,\HH)$ with $f=\sum_{h=0}^{n-1}\overline{x}^hf_h$ for some $f_0,\ldots,f_{n-1}\in\mc{SR}(\OO,\HH)$. For each $h\in\N^*$ and for each $I\in\sq_\HH$, we have $\debarI(\overline{x}^hf_h)_I=h\overline{x}^{h-1}f_h+\overline{x}^h\debarI f_I=h\overline{x}^{h-1}f_h$ on $\OO_I$, see Lemma \ref{lem:leibniz}$(\mr{ii})$ below for details. It follows that $\debarIn f_I=0$ on $\OO_I$ for all $I\in\sq_\HH$, i.e., $f\in\mc{SP}_n(\OO,\HH)$.

In \cite[Proposition 3.6]{ADS1}, the authors prove the converse inclusion $\mc{SP}_n(\OO,\HH)\subset\mc{SP}^n(\OO,\HH)$ by imitating the Gentili-Struppa `independence from $I$' argument we recalled in Remark \ref{rem:GS-indep-from-I}.

The next remark shows that, for each $n\geq2$, the latter inclusion $\mc{SP}_n(\OO,\HH)\subset\mc{SP}^n(\OO,\HH)$ is not correct; in particular, Theorem \ref{thm:poly-decomposition} is not correct. More precisely, we give examples of slice polyanalytic functions of each order $n\geq2$, which are not slice functions and do not admit any decomposition of the form \eqref{eq:decomposition}. In particular, it turns out that, in general, slice polyanalytic functions of each order $n\geq2$ do not satisfy neither the identity principle nor the representation formula.

\begin{remark}\label{1.14}
Let $v:\OO\to\HH$ be the function defined by
\begin{equation}
v(x):=-ixi\qquad\forall x\in\OO.
\end{equation}
Consider $I\in\sq_\HH$ and the restriction $v_I:\OO_I\to\HH$ of $v$ to $\OO_I$. Note that $v_I(\alpha+I\beta)=-i(\alpha+I\beta)i=\alpha-iIi\beta$ for all $\alpha,\beta\in\R$ with $\alpha+i\beta\in D$. We deduce that $\debarI v_I=\frac{1}{2}(1-IiIi)$ and $\debarI^{\,2}v_I=0$ on $\OO_I$. Since $v\in\mscr{C}^\omega(\OO,\HH)$ and $\debarI^{\,2}v_I=0$ on $\OO_I$ for all $I\in\sq_\HH$, it follows that
\begin{equation}
v\in\bigcap_{n\geq2}\mc{SP}_n(\OO,\HH).
\end{equation}

The function $v$ is not a (left) slice function. Otherwise, being $v_i(x)=-ixi=-i^2x=x$ for all $x\in\OO_i$, the identity principle for slice functions (which follows immediately from \eqref{eq:slice} and \eqref{eq:stem}) would imply that $v(x)=x$ or, equivalently, $xi=ix$ for all $x\in\OO$, which is a contradiction. A~similar argument proves that $v$ is not even a right slice function.

The identity principle is not valid for slice polyanalytic functions of each order $n\geq2$. Indeed, the function $v$ and the inclusion function $x:\OO\hookrightarrow\HH$ are distinct elements of $\bigcap_{n\geq2}\mc{SP}_n(\OO,\HH)$; however, they coincide on the whole $\OO_i$.

The function $v$ does not satisfy the representation formula in the following strong sense: given any $\alpha_0+i\beta_0\in D$ with $\alpha_0,\beta_0\in\R$ and $\beta_0>0$, there do not exist $a_0,b_0\in\HH$ such that $v(\alpha_0+I\beta_0)=a_0+Ib_0$ or, equivalently, $\alpha_0-iIi\beta_0=a_0+Ib_0$ for all $I\in\sq_\HH$. Indeed, if we set $I=\pm i$ in the latter equation, we deduce that $\alpha_0=a_0$ and $\beta_0=b_0$. On the other hand, if we set $I=j$, we obtain $-iji=j$, which is false because $-iji=-i(-ij)=i^2j=-j$. 

The function $v$ cannot be decomposed into the form \eqref{eq:decomposition}, i.e., it holds:
\begin{equation}\label{eq:not}
v\not\in\bigcup_{n\in\N^*}\mc{SP}^n(\OO,\HH).
\end{equation}
Suppose on the contrary that there exist $n\in\N^*$ and slice regular functions $f_0,\ldots,f_{n-1}:\OO\to\HH$ such that $v(x)=\sum_{h=0}^{n-1}\overline{x}^hf_h(x)$ for all $x\in\OO$. Without loss of generality, we can assume that $f_{n-1}$ is not constantly equal to $0$. It follows that $n\leq2$. Indeed, if $n$ would be $\geq3$, then $0=\debar_i^{\,n-3}(\debar_i^{\,2}v_i)=\debar_i^{\,n-1}v_i=(n-1)!f_{n-1}|_{\OO_i}$; hence $f_{n-1}$ would vanish on the whole $\OO$ by the identity principle for slice functions (recall that implication $(\mr{a})\Rightarrow(\mr{e})$ of Theorem \ref{thm:GP2013} implies that $f_{n-1}$ is a slice function). Thus, we can write $v(x)=f_0(x)+\overline{x}f_1(x)$ for all $x\in\OO$; in particular, $x_i=v_i(x_i)=f_0(x_i)+\overline{x_i}f_1(x_i)$ for all $x_i=\alpha+i\beta\in\OO_i$ with $\alpha,\beta\in\R$. Consequently, $0=\debar_i(x_i)=f_1(x_i)$ on $\OO_i$. Using the identity principle for slice functions again, we deduce that $f_1(x)=0$ and $f_0(x)=x$ for all $x\in\OO$. This implies that $-ixi=v(x)=f_0(x)+\overline{x}f_1(x)=x$ and hence $xi=ix$ for all $x\in\OO$, which is impossible. This proves \eqref{eq:not}.

Let $I\in\sq_\HH$. Since $\debarI^{\,2}v_I=0$ on $\OO_I$, Theorem \ref{thm:classical} ensures the existence of functions $f_{0,I},f_{1,I}:\OO_I\to\HH$ such that $\alpha-iIi\beta=v_I(x_I)=f_{0,I}(x_I)+\overline{x_I}f_{1,I}(x_I)$ for all $x_I=\alpha+I\beta\in\OO_I$ with $\alpha,\beta\in\R$, and $f_{0,I}$ and $f_{1,I}$ are holomorphic w.r.t. the complex structure induced by the left multiplication by $I$, i.e., $f_{0,I}$ and $f_{1,I}$ are of class $\mscr{C}^1$ and $\debarI f_{0,I}=\debarI f_{1,I}=0$ on $\OO_I$. Note that $f_{1,I}=\debarI v_I=\frac{1}{2}(1-IiIi)$, so
\[\textstyle
f_{0,I}(x_I)=\alpha-iIi\beta-(\alpha-I\beta)\frac{1}{2}(1-IiIi)=\alpha\frac{1}{2}(1+IiIi)+\beta\frac{1}{2}(I-iIi)=x_I\frac{1}{2}(1+IiIi).
\]
Consequently,
\[\textstyle
v_I(x_I)=x_I\frac{1}{2}(1+IiIi)+\overline{x_I}\frac{1}{2}(1-IiIi)
\]
for all $I\in\sq_\HH$ and for all $x_I\in\OO_I$. Note that the quaternions $\frac{1}{2}(1+IiIi)$ and $\frac{1}{2}(1-IiIi)$ depend on $I$; indeed, $IiIi=1$ if $I=i$ and $IiIi=-1$ if $I=j$. Thus, in this situation, the Gentili-Struppa `independence from $I$' argument does not work. 

Let $v_r:\OO\to\HH$ be the functions $v_r(x):=ix$. The function $v_r$ is an example of right slice regular function belonging to $\big(\bigcap_{n\geq2}\mc{SP}_n(\OO,\HH)\big)\setminus\big(\bigcup_{n\in\N^*}\mc{SP}^n(\OO,\HH)\big)$.

Besides some results of \cite{ADS1}, the examples described in this remark could also invalidate some results contained in papers that use \cite{ADS1}, as \cite{ACDS,ADS2}. \bs 
\end{remark}

{\it From now on, we will focus on global slice polyanalytic functions.}

Let $n\in\N^*$. If $k\in\{1,\ldots,n\}$ and $g\in\mscr{C}^n(\OO_*,\HH)$, then $\thetabar^k(g)\in\mscr{C}^{n-k}(\OO_*,\HH)$, where $\thetabar^k$ is the $k^{\mr{th}}$ power of $\thetabar$, i.e., $\thetabar^k=\thetabar\circ\ldots\circ\thetabar$ ($k$-times). In particular, the differential operator
\[
\thetabar^n:\mscr{C}^n(\OO_*,\HH)\to\mscr{C}^0(\OO_*,\HH),\quad g\mapsto\thetabar^n(g)
\]
is well-defined. For short, we write $\thetabar^kg$ in place of $\thetabar^k(g)$ for all $k\in\N^*$ and $g\in\mscr{C}^k(\OO_*,\HH)$.

Let $\mc{S}^n(\OO,\HH)$ be the set of slice functions from $\OO$ to $\HH$ induced by stem functions $F=F_1+\ui F_2:D\to\HH\otimes\C$ of class $\mscr{C}^n$ (i.e., $F_1,F_2:D\to\HH$ are of class $\mscr{C}^n$). Given $f=\I(F)\in\mc{S}^n(\OO,\HH)$, we define $\big(\frac{\partial}{\partial x^c}\big)^nf:\OO\to\HH$ by
\[
\left(\frac{\partial}{\partial x^c}\right)^nf:=\I\left(\frac{\partial^nF}{\partial\overline{z}^n}\right).
\]
Evidently, it holds $\big(\frac{\partial}{\partial x^c}\big)^nf=\big(\frac{\partial}{\partial x^c}\big)\big(\big(\frac{\partial}{\partial x^c}\big)^{n-1}f\big)$, where $\big(\frac{\partial}{\partial x^c}\big)^0f:=f$.

We have:

\begin{lemma}\label{lem:reg-n}
Let $n\in\N$ and let $f\in\mc{S}^n(\OO,\HH)$. Then $f\in\mscr{C}^n_\sbs(\OO,\HH)$ and $f_*\in\mscr{C}^n(\OO_*,\HH)$.
\end{lemma}

By the latter result, if $f\in\mc{S}^n(\OO,\HH)$ for some $n\in\N^*$, then $\debarIn f_I$ and $\thetabar^n f_*$ make sense.

Our next result extends Theorems \ref{thm:GP2013} and \ref{thm:GP2013-sbs} to global slice polyanalytic functions of any order. 

\begin{theorem}\label{main}
Let $n\in\N^*$ and let $f:\OO\to\HH$ be a function. The following assertions are equivalent.
\begin{itemize}
 \item[$(\mr{a})$] $f\in\mc{SP}^n(\OO,\HH)$, i.e., $f$ is global slice polyanalytic of order $n$.
 \item[$(\mr{b})$] $f\in\mc{S}^{n-1}(\OO,\HH)$, $f_*\in\mscr{C}^n(\OO_*,\HH)$ and $f_*\in\ker(\thetabar^n)$.
 \item[$(\mr{c})$] $f\in\mc{S}^n(\OO,\HH)$ and $\debarIn f_I=0$ on $\OO_I$ for all $I\in\sq_\HH$.
 \item[$(\mr{c}')$] $f\in\mc{S}^n(\OO,\HH)$ and $f_*\in\ker(\thetabar^n)$.
 \item[$(\mr{c}'')$] $f\in\mc{S}^n(\OO,\HH)$ and $\big(\frac{\partial}{\partial x^c}\big)^nf=0$ on $\OO$.
\end{itemize} 
\end{theorem}

As a consequence, we obtain:

\begin{corollary}\label{cor:main}
For each $n\in\N^*$, it holds:
\[
\mc{SP}^n(\OO,\HH)=\{f\in\mc{S}^n(\OO,\HH):f_*\in\ker(\thetabar^n)\}\subset\mscr{C}^\omega(\OO,\HH).
\]
\end{corollary}


\subsection{The monogenic case}\label{subsec:monogenic}

The results on quaternions presented above can be suitably restated in the monogenic case.

{\it Fix $m\in\N$ with $m\geq2$, and consider the Clifford algebra $\R_m=C\ell_{0,m}$.}

Let $\{e_K\}_{K\in\pa(m)}$ be the standard real vector basis of $\R_m=\R^{2^m}$, where $\pa(m)$ denotes the power set of $\{1,\dots,m\}$ and $e_\emptyset$ is the unity $1$ of $\R_m$. For each $j\in\{1,\ldots,m\}$, we write $e_j$ in place of $e_{\{j\}}$. Let $\R^{m+1}$ be the real vector subspace of $\R_m$ consisting of paravectors $x=x_0+\sum_{j=1}^mx_je_j$ with $x_0,\ldots,x_m\in\R$. For each $x=x_0+\sum_{j=1}^mx_je_j\in\R^{m+1}$, set $\mr{Re}(x):=x_0$, $\mr{Im}(x):=x-\mr{Re}(x)=\sum_{j=1}^nx_je_j$ and $\overline{x}:=\mr{Re}(x)-\mr{Im}(x)=x_0-\sum_{j=1}^nx_je_j$. Let $\sq^{m-1}$ be the $(m-1)$-sphere of paravector imaginary units, i.e.
\[\textstyle
\sq^{m-1}=\{x\in\R^{m+1}:x^2=-1\}=\big\{x\in\R^{m+1}:\mr{Re}(x)=0,|\mr{Im}(x)|=1\big\},
\]
where $|\mr{Im}(x)|=(\sum_{j=1}^mx_j^2)^{1/2}$ is the Euclidean norm of $\mr{Im}(x)=\sum_{j=1}^mx_je_j$ in $\R^{m+1}$. Given a subset $E$ of $\C$, we define the circularization $\OO_{m,E}$ of $E$ in $\R^{m+1}$ by
\[
\OO_{m,E}:=\{\alpha+I\beta\in\R^{m+1}:\alpha,\beta\in\R,\alpha+i\beta\in E, I\in\sq^{m-1}\}.
\]

A subset $S$ of $\R^{m+1}$ is called circular, or axially symmetric, if $S=\OO_{m,E}$ for some $E\subset\C$. 

{\it Fix a non-empty connected circular open subset $\OO$ of $\R^{m+1}$ such that $\OO\cap\R\neq\emptyset$ and denote by $D$ the non-empty open subset of $\C$ with $D\cap\R\neq\emptyset$ such that $\OO=\OO_{m,D}$.}

Let $I\in\sq^{m-1}$ and let $\C_I:=\{\alpha+I\beta\in\R^{m+1}:\alpha,\beta\in\R\}$. We define the set $\OO_I$, the real analytic isomorphism $\phi_I:D\to\OO_I$ and, given any function $f:\OO\to\R_m$, the restriction $f_I:\OO_I\to\R_m$ as in \eqref{def:OO_I}, \eqref{def:phi_I} and \eqref{def:f_I}, respectively. Moreover, given $g:\OO_I\to\R_m$ of class $\mscr{C}^n$ for some $n\in\N^*$, we define $\debarIn g:\OO_I\to\R_m$ as in \eqref{def:debarIn}. Let $O$ be a non-empty open subset of $\R_m$. For each $n\in\N^*$, we denote by $\mscr{C}^n(O,\R_m)$ the set of functions of class $\mscr{C}^n$ from $O$ to $\R_m$. We denote by $\mscr{C}^\omega(O,\R_m)$ the set of real analytic functions from $O$ to $\HH$. Furthermore, we define $\mscr{C}^0(O,\R_m)$ as the set of continuous functions from $O$ to~$\R_m$.

\begin{definition}[{\cite[Definition 2.1]{CoSaSt2009Israel}}]\label{def:SM}
A function $f:\OO\to\R_m$ is called \emph{slice monogenic} if $f\in\mscr{C}^1(\OO,\R_m)$ and, for each $I\in\sq^{m-1}$, it holds:
\[
\debarI f_I=0 \quad\text{on $\OO_I$.}
\]
The set of slice monogenic functions from $\OO$ to $\R_m$ is denoted by $\mc{SM}(\OO,\R_m)$. \bs
\end{definition}

Let $\R_m\to\R_m$, $x\mapsto x^c$ be the Clifford conjugation on $\R_m$ (see \cite[Definition 3.7, p. 56]{GHS}), let $\sq_{\R_m}:=\{x\in\R_m:x^2=-1,x^c=-x\}$ and, for each $I\in\sq_{\R_m}$, let $\C_I:=\{\alpha+I\beta\in\R_m:\alpha,\beta\in\R\}$. Note that $x^c=\overline{x}$ for all $x\in\R^{m+1}$. The quadratic cone $Q_{\R_m}$ of $\R_m$ is defined by $Q_{\R_m}:=\bigcup_{I\in\sq_{\R_m}}\C_I$. The circularization $\OO_D$ of $D$ in $\R_m$ is defined by
\begin{equation}\label{eq:OO_D}
\OO_D:=\{\alpha+I\beta\in\R_m:\alpha,\beta\in\R,\alpha+i\beta\in D, I\in\sq_{\R_m}\}.
\end{equation}
The set $\Omega_D$ turns out to be open in $Q_{\R_m}$. Since we assumed that $m\geq2$, we have that $e_{\{1,2\}}\in\sq_{\R_m}\setminus\sq^{m-1}$. It follows that
\[
\sq^{m-1}\subsetneqq\sq_{\R_m}, \quad \R^{m+1}\subsetneqq Q_{\R_m}\quad\text{and}\quad\OO=\OO_{m,D}\subsetneqq\OO_D.
\]

A function $F=F_1+\ui F_2:D\to\R_m\otimes_\R\C$ is a stem function if it satisfies the even-odd properties \eqref{eq:stem1}. If $F=F_1+\ui F_2:D\to\R_m\otimes_\R\C$ is a stem function and $f:\OO\to\R_m$ is the function defined by equation \eqref{eq:slice} for all $I\in\sq^{m-1}$, then we say that $f:\OO\to\R_m$ is a \emph{(left) slice function}. If this is the case, then we say that $f$ is induced by $F$ and we write $f=\I(F)$. Also in this case, $F$ is uniquely determined by $f$ via formula \eqref{eq:stem}. Moreover, the slice function $f:\OO\to\R_m$ extends uniquely to a (left) slice function $\widetilde{f}:\OO_D\to\R_m$ in the standard sense of \cite[Definition 8]{AIM2011}, by requiring that $\widetilde{f}$ satisfies equation \eqref{eq:slice} for all $I\in\sq_{\R_m}$, i.e., $\widetilde{f}(\alpha+I\beta)=F_1(\alpha+i\beta)+IF_2(\alpha+i\beta)$ for all $\alpha,\beta\in\R$ with $\alpha+i\beta\in D$ and for all $I\in\sq_{\R_m}$. We refer again to \cite{AIM2011} for further details on slice and slice regular functions on $\OO_D$.

If $F$ is of class $\mscr{C}^1$, we can define the slice function $\frac{\partial f}{\partial x^c}:\OO\to\R_m$ as in \eqref{def:slicedebar}. 

As in \eqref{def:OO_*} and \eqref{def:f_*}, we set $\OO^*:=\OO\setminus\R$ and, given any function $f:\OO\to\R_m$, we define $f_*:\OO_*\to\R_m$ by $f_*:=f|_{\OO_*}$.

\begin{definition}\label{def:thetabarmono}
We define the differential operator $\thetabarm:\mscr{C}^1(\OO_*,\R_m)\to\mscr{C}^0(\OO_*,\R_m)$ by
\[
\thetabarm:=\frac{1}{2}\left(\frac{\partial}{\partial x_0}+\frac{\mr{Im}(x)}{|\mr{Im}(x)|^2}\left(\sum_{h=1}^mx_h\frac{\partial}{\partial x_h}\right)\right).\,\text{ \bs}
\]
\end{definition}

As in \cite[Definition 1.6]{CGS2013}, we define the differential operator $G_m:\mscr{C}^1(\OO,\R_m)\to\mscr{C}^0(\OO,\R_m)$ by
\begin{equation}\label{def:Gm}
G_m:=|\mr{Im}(x)|^2\frac{\partial}{\partial x_0}+\mr{Im}(x)\left(\sum_{h=1}^mx_h\frac{\partial}{\partial x_h}\right).
\end{equation}
Evidently, if $f\in\mscr{C}^1(\OO,\R_m)$ and $x\in\OO_*$, then $G(f)(x)=2|\mr{Im}(x)|^2\,\thetabarm(f_*)(x)$.

As in the quaternionic case, given a function $f:\OO\to\R_m$, we say that $f$ is \emph{slice-by-slice continuous} if $f_I:\OO_I\to\R_m$ is continuous for all $I\in\sq^{m-1}$. We denote by $\mscr{C}^0_\sbs(\OO,\R_m)$ the set of slice-by-slice continuous functions from $\OO$ to $\R_m$. Given $n\in\N^*$, we say also that $f:\OO\to\R_m$ is a s\emph{lice-by-slice $\mscr{C}^n$} function if $f_I\circ\phi_I:D\to\R_m$ is of class $\mscr{C}^n$ for all $I\in\sq^{m-1}$. We denote by $\mscr{C}^n_\sbs(\OO,\R_m)$ the set of slice-by-slice $\mscr{C}^n$ functions from $\OO$ to $\R_m$. There exist functions $f:\OO\to\R_m$ such that $f\in\mscr{C}^n_\sbs(\OO,\R_m)$ and $f_*\in\mscr{C}^n(\OO_*,\R_m)$ for all $n\in\N^*$; however, $f\not\in\mscr{C}^0(\OO,\R_m)$: it suffices to define $f$ as in \eqref{eq:sbscont-noncont} replacing the expression `$\,x_1^2x_2(x_1^4+x_2^2+x_3^2)^{-1}\,$' with `$\,x_1^2x_2(x_1^4+\sum_{h=2}^mx_h^2)^{-1}\,$'. 

Theorems \ref{thm:GP2013} and \ref{thm:GP2013-sbs}, and Corollary \ref{cor:GP2013} extend to the monogenic case. 

\begin{theorem}\label{thm:GP2013mono}
Let $f:\OO\to\R_m$ be a function. The following assertions are equivalent.
\begin{itemize}
 \item[$(\mr{a})$] $f\in\mc{SM}(\OO,\HH)$, i.e., $f$ is slice monogenic.
  \item[$(\mr{a}')$] $f\in\mscr{C}^1_\sbs(\OO,\R_m)$ and $\debarI f_I=0$ on $\OO_I$ for all $I\in\sq^{m-1}$.
 \item[$(\mr{b})$] $f\in\mscr{C}^0(\OO,\R_m)$, $f_*\in\mscr{C}^1(\OO_*,\R_m)$ and $f_*\in\ker(\thetabarm)$.
 \item[$(\mr{b}')$] $f\in\mscr{C}^0_\sbs(\OO,\R_m)$, $f_*\in\mscr{C}^1(\OO_*,\R_m)$ and $f_*\in\ker(\thetabarm)$.
 \item[$(\mr{c})$] $f\in\mscr{C}^1(\OO,\R_m)$ and $f_*\in\ker(\thetabarm)$.
 \item[$(\mr{d})$] $f\in\mscr{C}^1(\OO,\R_m)$ and $f\in\ker(G_m)$.
 \item[$(\mr{e})$] $f\in\mc{S}^1(\OO,\HH)$ and $\frac{\partial f}{\partial x^c}=0$ on $\OO$.
\end{itemize} 
\end{theorem}

\begin{corollary}
$\mc{SM}(\OO,\HH)=\{f\in\mscr{C}^1(\OO,\R_m):f_*\in\ker(\thetabarm)\}=\ker(G_m)\subset\mscr{C}^\omega(\OO,\R_m)$. 
\end{corollary}

\begin{remark}
As in \cite[Definition 1.6]{CGS2013}, we define $\mc{GM}(\OO)$ as the set of all distributional solutions $f$ of the differential equation $G_m(f)=0$ on $\OO$. By \cite[Theorem 3.5]{CGS2013}, we know that $\mc{SM}(\OO,\R_m)\subsetneqq\mc{GM}(\OO)$ (see also \cite{CS2014}). Similarly to Remark \ref{rem:distribution}, equivalence $(\mr{a})\Leftrightarrow(\mr{b})$ of Theorem \ref{thm:GP2013mono} and the equation `$G(f)(x)=2|\mr{Im}(x)|^2\,\thetabarm(f_*)(x)$ for all $x\in\OO_*$' imply that $\mc{SM}(\OO,\R_m)$ is the set of functions $f$ in $\mscr{C}^0(\OO,\R_m)$ such that $f_*$ belongs to $\mscr{C}^1(\OO_*,\R_m)$ and $G_m(f)=0$ on $\OO$ in the sense of distributions, i.e., $f\in\mc{GM}(\OO)$. \bs
\end{remark}

\begin{definition}[{\cite[Definition 5.1]{ADS2}}]
Let $n\in\N^*$. A function $f:\OO\to\R_m$ is said to be \emph{slice polymonogenic of order $n$} if $f\in\mscr{C}^n(\OO,\R_m)$ and, for each $I\in\sq^{m-1}$, it holds:
\[
\debarIn f_I=0 \quad\text{on $\OO_I$.}
\]
The set of slice polymonogenic functions of order $n$ from $\OO$ to $\R_m$  is denoted by $\mc{SM}_n(\OO,\R_m)$, i.e.,
\[
\text{$\mc{SM}_n(\OO,\R_m):=\{f\in\mscr{C}^n(\OO,\R_m):\debarIn f_I(x)=0\;\; \forall I\in\sq^{m-1},\forall x\in\OO_I\}$. \bs}
\]
\end{definition}

Similarly to the quaternionic case, we give the following definition.

\begin{definition}\label{def:sl-polymonogenic-global}
Given $n\in\N^*$ and a function $f:\OO\to\R_m$, we say that $f$ is a \emph{global slice polymonogenic function of order $n$} if there exist $f_0,\ldots,f_{n-1}\in\mc{SM}(\OO,\R_m)$ such that $f(x)=\sum_{h=0}^{n-1}\overline{x}^hf_h(x)$ for all $x\in\OO$. We denote by $\mc{SM}^n(\OO,\R_m)$ the set of global slice polymonogenic functions of order $n$ from $\OO$ to $\R_m$. \bs
\end{definition}

Evidently, it holds $\mc{SM}(\OO,\R_m)=\mc{SM}_1(\OO,\R_m)=\mc{SM}^1(\OO,\R_m)$. Theorem 5.4 of \cite{ADS2} asserts that the equality $\mc{SM}_n(\OO,\R_m)=\mc{SM}^n(\OO,\R_m)$ holds also for all $n\geq2$. As in the quaternionic slice polyanalytic case, Leibniz's rule for $\debarI$ and the slice monogenicity of the $f_h$'s imply that $\mc{SM}^n(\OO,\R_m)\subset\mc{SM}_n(\OO,\R_m)$ for all $n\geq2$.  However, by the same arguments used in Remark \ref{1.14}, we see that the function $v_m:\OO\to\R_m$, defined by $v_m(x):=-e_1xe_1$, belongs to $\big(\bigcap_{n\geq2}\mc{SM}_n(\OO,\R_m)\big)\setminus\big(\bigcup_{n\in\N^*}\mc{SM}^n(\OO,\R_m)\big)$.  

For each $n\in\N^*$, the $n^{\mr{th}}$ power $\thetabarm^{\,n}$ of $\thetabarm$ is a well-defined operator from $\mscr{C}^n(\OO_*,\R_m)$ to $\mscr{C}^0(\OO_*,\R_m)$. For short, we write $\thetabarm^{\,k}g$ in place of $\thetabarm^{\,k}(g)$ for all $k\in\N^*$ and $g\in\mscr{C}^k(\OO_*,\R_m)$. Denote by $\mc{S}^n(\OO,\R_m)$ the set of slice functions from $\OO$ to $\R_m$ induced by stem functions $F:D\to\R_m\otimes_\R\C$ of class $\mscr{C}^n$. If $f=\I(F)\in\mc{S}^n(\OO,\R_m)$, then we define $\big(\frac{\partial}{\partial x^c}\big)^nf:\OO\to\R_m$ by $\big(\frac{\partial}{\partial x^c}\big)^nf:=\I\big(\frac{\partial^nF}{\partial\overline{z}^n}\big)$. As in Lemma \ref{lem:reg-n}, if $f\in\mc{S}^n(\OO,\R_m)$ for some $n\in\N^*$, then $f\in\mscr{C}^n_\sbs(\OO,\R_m)$ and $f_*\in\mscr{C}^n(\OO_*,\R_m)$.

Also Theorem \ref{main} and its corollary extend to the monogenic case. 

\begin{theorem}\label{mainmono}
Let $n\in\N^*$ and let $f:\OO\to\R_m$ be a function. The following assertions are equivalent.
\begin{itemize}
 \item[$(\mr{a})$] $f\in\mc{SM}^n(\OO,\R_m)$., i.e., $f$ is global slice polymonogenic of order $n$.
 \item[$(\mr{b})$] $f\in\mc{S}^{n-1}(\OO,\R_m)$, $f_*\in\mscr{C}^n(\OO_*,\R_m)$ and $f_*\in\ker(\thetabarm^{\,n})$.
 \item[$(\mr{c})$] $f\in\mc{S}^n(\OO,\R_m)$ and $\debarIn f_I=0$ on $\OO_I$ for all $I\in\sq^{m-1}$.
 \item[$(\mr{c}')$] $f\in\mc{S}^n(\OO,\R_m)$ and $f_*\in\ker(\thetabarm^{\,n})$.
 \item[$(\mr{c}'')$] $f\in\mc{S}^n(\OO,\R_m)$ and $\big(\frac{\partial}{\partial x^c}\big)^nf=0$ on $\OO$.
\end{itemize} 
\end{theorem}

\begin{corollary}
For each $n\in\N^*$, it holds:
\[
\mc{SM}^n(\OO,\R_m)=\{f\in\mc{S}^n(\OO,\R_m):f_*\in\ker(\thetabarm^{\,n})\}\subset\mscr{C}^\omega(\OO,\R_m).
\]
\end{corollary}


\section{Proofs}\label{sec:proofs}

We give the proofs of the quaternionic results only. The proofs in the monogenic case are similar.

We begin with a remark.

\begin{remark}
Let $E$ be a non-empty open subset of $\C$ invariant under complex conjugation and let $\OO_E$ be the circularization of $E$ in $\HH$; for instance, $E=D$ and $\OO_E=\OO_D=\OO$ or $E=D\setminus\R$ and $\OO_E=\OO_{D\setminus\R}=\OO_*$. Note that $\OO_E$ is open in $\HH$; indeed, the function $\zeta:\HH\to\C$ defined by $\zeta(x):=\mr{Re}(x)+i|\mr{Im}(x)|$ is continuous and $\OO_E=\zeta^{-1}(E)$. Consider a function $f:\OO_E\to\HH$. Let us extend to such a function $f$ the definitions of slice-by-slice continuous function and of slice-by-slice $\mscr{C}^n$ function for $n\in\N^*$ in the natural way.

We say that $f:\OO_E\to\HH$ is \emph{slice-by-slice continuous} if, for all $I\in\sq_\HH$, the restriction $f_I$ of $f$ to $\OO_{E,I}:=\OO_E\cap\C_I$ is continuous. We denote by $\mscr{C}^0_\sbs(\OO_E,\HH)$ the set of slice-by-slice continuous functions from $\OO_E$ to $\HH$. Given any $n\in\N^*$, we say that $f:\OO_E\to\HH$ is a \emph{slice-by-slice $\mscr{C}^n$} function if, for all $I\in\sq_\HH$, the composition $f_I\circ\phi_I:E\to\HH$ is of class $\mscr{C}^n$, where $\phi_I:E\to\OO_{E,I}$ is the real analytic isomorphism given by $\phi_I(\alpha+i\beta):=\alpha+I\beta$ for all $\alpha,\beta\in\R$ with $\alpha+i\beta\in E$. We denote by $\mscr{C}^n_\sbs(\OO_E,\HH)$ the set of slice-by-slice $\mscr{C}^n$ functions from $\OO_E$ to $\HH$. Note that, if $f\in\mscr{C}^n_\sbs(\OO_E,\HH)$ and $I\in\sq_\HH$, then we can define $\debarIn f_I:\OO_{E,I}\to\HH$ by means of \eqref{def:debarIn} with $f_I$ in place of $g$. Furthermore, if $\OO_E\cap\R\neq\emptyset$, $u$ is a function in $\mscr{C}^n_\sbs(\OO_E,\HH)$ and $u_*:\OO_E\setminus\R\to\HH$ is the restriction of $u$ to $\OO_E\setminus\R$, then $\OO_E\setminus\R=\OO_{E\setminus\R}$, $u_*\in\mscr{C}^n_\sbs(\OO_E\setminus\R,\HH)$, $\OO_{E,I}\setminus\R$ is open in $\OO_{E,I}$ and hence $\big(\debarIn u_I\big)|_{\OO_{E,I}\setminus\R}=\debarIn (u_*)_I$. \bs
\end{remark}

\subsection{Proofs of Theorem \ref{thm:GP2013} and Corollary \ref{cor:GP2013}}

\begin{lemma}\label{lem:1}
Let $n\in\N^*$, let $f\in\mscr{C}^n(\OO_*,\HH)$ and let $I\in\sq_\HH$. Then $(\thetabar^nf)_I=\debarIn f_I$ on $\OO_I\setminus\R$ or, equivalently,
\begin{equation}\label{eq:bar}
\thetabar^nf(x)=\debarIn f_I(x)\qquad\forall x\in\OO_I\setminus\R.
\end{equation}
\end{lemma}
\begin{proof}
Let us proceed by induction on $n\in\N^*$.

First, we consider the case $n=1$. Actually, this case coincides with \cite[Theorem 2.2(i)]{global2013}. However in \cite{global2013}, for short, some details of the proof of Theorem 2.2(i) were omitted. Here we give the proof in detail. Write $I$ in coordinates: $I=i_1i+i_2j+i_3k$ with $i_1,i_2,i_3\in\R$. Define $\widehat{f}_I:D\setminus\R\to\HH=\R^4$ by $\widehat{f}_I(\alpha+i\beta):=f(\alpha+I\beta)$ for all $\alpha,\beta\in\R$ with $\alpha+i\beta\in D\setminus\R$. Consider a point $x=a+Ib\in\OO_I\setminus\R$ with $a,b\in\R$, and set $z:=a+ib\in D\setminus\R$.

First, {\it suppose $b>0$}. Note that $\frac{\partial\widehat{f}_I}{\partial\alpha}(z)$ and $\frac{\partial\widehat{f}_I}{\partial\beta}(z)$ equal the derivatives of $f$ at $x$ in the direction $1$ and $I$, respectively. Thus, we have
\begin{equation}\label{eq:direction1}
\frac{\partial\widehat{f}_I}{\partial\alpha}(z)=\frac{\partial f}{\partial x_0}(x)
\end{equation}
and
\begin{equation}\label{eq:directionI}
\frac{\partial\widehat{f}_I}{\partial\beta}(z)=\sum_{h=1}^3i_h\frac{\partial f}{\partial x_h}(x).
\end{equation}
Since $b>0$, it holds $|\mr{Im}(x)|=b$, $I=|\mr{Im}(x)|^{-1}\mr{Im}(x)$ and $i_h=|\mr{Im}(x)|^{-1}x_h$ for all $h\in\{1,2,3\}$, where $x=x_0+x_1i+x_2j+x_3k$ with $x_0,x_1,x_2,x_3\in\R$. Consequently, \eqref{eq:directionI}can be rewritten as
\begin{equation}\label{eq:directionI-bis}
\frac{\partial\widehat{f}_I}{\partial\beta}(z)=|\mr{Im}(x)|^{-1}\sum_{h=1}^3x_h\frac{\partial f}{\partial x_h}(x).
\end{equation}
Recall that, by definition, $\debarI f_I(x)=\frac{1}{2}\big(\frac{\partial\widehat{f}_I}{\partial\alpha}(z)+I\frac{\partial\widehat{f}_I}{\partial\beta}(z)\big)$. Combining the latter equality with \eqref{eq:direction1} and \eqref{eq:directionI-bis}, we obtain
\begin{equation*}\label{eq:beta>0}
\text{$\debarI f_I(x)=\thetabar f(x)$.}
\end{equation*}

{\it Suppose $b<0$}. Since $x=a+(-I)(-b)$ and $-b>0$, we can repeat the preceding argument (with $z$ replaced with its conjugate $\overline{z}$) to the function $f_{-I}:\OO_{-I}\setminus\R=\OO_I\setminus\R\to\HH$, obtaining
\begin{equation*}\label{eq:beta<0}
\text{$\debar_{-I} f_{-I}(x)=\thetabar f(x)$.}
\end{equation*}
Moreover, if $\widehat{f}_{-I}:D\setminus\R\to\HH$ is the function defined by $\widehat{f}_{-I}(\alpha+i\beta):=f(\alpha-I\beta)$ for all $\alpha,\beta\in\R$ with $\alpha+i\beta\in D$, then $\widehat{f}_{-I}(\alpha+i\beta)=\widehat{f}_I(\alpha-i\beta)$ so $\frac{\partial\widehat{f}_{-I}}{\partial\alpha}(\overline{z})=\frac{\partial\widehat{f}_I}{\partial\alpha}(z)$ and $\frac{\partial\widehat{f}_{-I}}{\partial\beta}(\overline{z})=-\frac{\partial\widehat{f}_I}{\partial\beta}(z)$. It follows that
\[
\thetabar f(x)=\debar_{-I} f_{-I}(x)=\frac{1}{2}\left(\frac{\partial\widehat{f}_{-I}}{\partial\alpha}(\overline{z})+(-I)\frac{\partial\widehat{f}_{-I}}{\partial\beta}(\overline{z})\right)=\frac{1}{2}\left(\frac{\partial\widehat{f}_I}{\partial\alpha}(z)+I\frac{\partial\widehat{f}_I}{\partial\beta}(z)\right)=\debarI f_I(x).
\]

This completes the proof of the case $n=1$, i.e., $(\thetabar f)_I=\debarI f_I$ for all $f\in\mscr{C}^1(\OO_*,\HH)$.

Let $n\geq2$ and let $f\in\mscr{C}^n(\OO_*,\HH)\subset\mscr{C}^{n-1}(\OO_*,\HH)$. By induction hypothesis, we have that $(\thetabar^{n-1}f)_I=\debarI^{\,n-1}f_I$. On the other hand, $\thetabar^{n-1}f\in\mscr{C}^1(\OO_*,\HH)$ so by the case $n=1$ we have
\[
(\thetabar^nf)_I=(\thetabar(\thetabar^{n-1}f))_I=\debarI(\thetabar^{n-1}f)_I=\debarI(\debarI^{\,n-1}f_I)=\debarI^{\,n}f_I.
\]
The proof is complete.
\end{proof}

\begin{proof}[Proof of Lemma \ref{lem:reg-n}]
Let $n\in\N$, let $f\in\mc{S}^n(\OO,\HH)$ and let $F=F_1+iF_2:D\to\HH\otimes_\R\C$ be the stem function of class $\mscr{C}^n$ inducing $f$ (by the term `of class $\mscr{C}^0$', we mean `continuous'). Given any $I\in\sq_\HH$, we have that $f_I\circ\phi_I=F_1+IF_2$ on $D$. Thus, $f_I\circ\phi_I$ is of class $\mscr{C}^n$ for all $I\in\sq_\HH$, i.e., $f\in\mscr{C}^n_\sbs(\OO,\HH)$. Define the real analytic functions $\zeta:\OO_*\to D$ and $\mscr{I}:\OO_*\to\HH$ by
\[
\zeta(x):=\mr{Re}(x)+i|\mr{Im}(x)|
\quad\text{and}\quad
\mscr{I}(x):=|\mr{Im}(x)|^{-1}\mr{Im}(x).
\]
Since $f_*(x)=(F_1\circ\zeta)(x)+\mscr{I}(x)(F_2\circ\zeta)(x)$ for all $x\in\OO_*$, it follows that $f_*\in\mscr{C}^n(\OO_*,\HH)$. 
\end{proof}

\begin{remark}\label{rem:reg-n}
The preceding proof ensures that, if $f=\I(F):\OO\to\HH$ is a slice function and $F|_{D\setminus\R}$ is of class~$\mscr{C}^n$ for some $n\in\N^*$, then $f_*\in\mscr{C}^n(\OO_*,\HH)$. \bs 
\end{remark}

\begin{lemma}\label{lem:de-n}
Let $n\in\N^*$ and let $f\in\mc{S}^n(\OO,\HH)$. Then $f\in\mscr{C}^n_\sbs(\OO_I,\HH)$ and, for each $I\in\sq_\HH$, it holds
\begin{equation}\label{eq:de-n}
\left(\frac{\partial}{\partial x^c}\right)^nf(x)=\debarIn f_I(x) \quad\text{for all $x\in\OO_I$}.
\end{equation} 
\end{lemma}
\begin{proof}
By Lemma \ref{lem:reg-n}, we know that $f\in\mscr{C}^n_\sbs(\OO_*,\HH)$. Let $F=F_1+iF_2:D\to\HH\otimes_\R\C$ be the stem function of class $\mscr{C}^n$ inducing $f$. Recall that, given any $I\in\sq_\HH$, if $z=\alpha+i\beta\in D$ with $\alpha,\beta\in\R$ and $z_I:=\alpha+I\beta\in\OO_I$, then $f_I(z_I)=F_1(z)+IF_2(z)$.

Let us prove \eqref{eq:de-n} by induction on $n\in\N^*$. Bearing in mind~\eqref{def:debarF}, we have:
\begin{align*}
\debarI f_I(z_I)&=\textstyle\frac{1}{2}\left(\frac{\partial (F_1+IF_2)}{\partial\alpha}(z)+I\frac{\partial (F_1+IF_2)}{\partial\beta}(z)\right)=\\
&\textstyle=\frac{1}{2}\left(\frac{\partial F_1}{\partial\alpha}(z)-\frac{\partial F_2}{\partial\beta}(z)+I\left(\frac{\partial F_1}{\partial\beta}(z)+\frac{\partial F_2}{\partial\alpha}(z)\right)\right)=\frac{\partial f}{\partial x^c}(z_I),
\end{align*}
i.e., $\big(\big(\frac{\partial}{\partial x^c}\big)f\big)_I=\debarI f_I$ on $\OO_I$. This proves the case $n=1$. Let $n\geq2$. By induction hypothesis, we can assume that $\big(\big(\frac{\partial}{\partial x^c}\big)^{n-1}f\big)_I=\debarI^{\,n-1} f_I$ on $\OO_I$. Since $\big(\frac{\partial}{\partial x^c}\big)^{n-1}f\in\mc{S}^1(\OO,\HH)$, thanks to the case $n=1$, we have:
\[\textstyle
\left(\left(\frac{\partial}{\partial x^c}\right)^nf\right)_I=\left(\left(\frac{\partial}{\partial x^c}\right)\left(\left(\frac{\partial}{\partial x^c}\right)^{n-1}f\right)\right)_I=\debarI\left(\left(\frac{\partial}{\partial x^c}\right)^{n-1}f\right)_I=\debarI\left(\debarI^{\,n-1} f_I\right)=\debarIn f_I.
\]
This completes the proof.
\end{proof}

In the case $n=1$, a version of the latter result is contained in \cite[Remark 1.7]{global2013}.

\begin{lemma}\label{lem:real-analytic}
Let $f:\OO\to\HH$ be a slice regular function in the sense of \cite[Definition 8]{AIM2011}, i.e., $f\in\mc{S}^1(\OO,\HH)$ and $\frac{\partial f}{\partial x^c}=0$ on $\OO$. Then $f\in\mscr{C}^\omega(\OO,\HH)$.
\end{lemma}
\begin{proof}
Since $\frac{\partial f}{\partial x^c}=\I(\frac{\partial F}{\partial\overline{z}})=0$ on $\OO$, equations \eqref{eq:stem} (applied with a fixed $I\in\sq_\HH$) imply that $\frac{\partial F}{\partial\overline{z}}=0$ on $D$, i.e., $F$ is holomorphic. In particular, $F$ is real analytic. By \cite[Proposition 7(3)]{AIM2011}, it follows that $f\in\mscr{C}^\omega(\OO,\HH)$.
\end{proof}

\begin{proof}[Proof of Theorem \ref{thm:GP2013}]
Implication $(\mr{a})\Rightarrow(\mr{c})$ follows immediately from  \cite[Theorem 2.2(i)]{global2013} or above Lemma \ref{lem:1} with $n=1$. Implication $(\mr{c})\Rightarrow(\mr{b})$ is evident. Equivalence $(\mr{b})\Leftrightarrow(\mr{e})$ coincides with \cite[Theorem 2.4]{global2013}. Assertion \eqref{eq:thetabarG} implies at once equivalence $(\mr{c})\Leftrightarrow(\mr{d})$. It remains to show implication $(\mr{e})\Rightarrow(\mr{a})$. Let $f:\OO\to\HH$ be a function satisfying $(\mr{e})$, i.e., $f$ is a slice regular function in the sense of \cite[Definition 8]{AIM2011}. By Lemma \ref{lem:real-analytic}, we know that $f\in\mscr{C}^\omega(\OO,\HH)\subset\mscr{C}^1(\OO,\HH)$. Now Lemma \ref{lem:de-n} (with $n=1$) ensures that $\debarI f_I=0$ on $\OO_I$ for all $I\in\sq_\HH$. This proves that $f\in\mc{SR}(\OO,\HH)$.
\end{proof}

\begin{proof}[Proof of Corollary \ref{cor:GP2013}]
Equivalences $(\mr{a})\Leftrightarrow(\mr{c})\Leftrightarrow(\mr{d})$ of Theorem \ref{thm:GP2013} immediately imply that $\mc{SR}(\OO,\HH)=\{f\in\mscr{C}^1(\OO,\HH):f_*\in\ker(\thetabar)\}=\ker(G)$. The inclusion $\mc{SR}(\OO,\HH)\subset\mscr{C}^\omega(\OO,\HH)$ was proved in Lemma \ref{lem:real-analytic}.
\end{proof}

\begin{lemma}\label{lem:r-a-spn}
$\mc{SP}^n(\OO,\HH)\subset\mscr{C}^\omega(\OO,\HH)$ for all $n\in\N^*$.
\end{lemma}
\begin{proof}
Let $n\in\N^*$ and let $f=\sum_{h=0}^{n-1}\overline{x}^hf_h\in\mc{SP}^n(\OO,\HH)$ for some $f_0,\ldots,f_{n-1}\in\mc{SR}(\OO,\HH)$. By implication $(\mr{a})\Rightarrow(\mr{e})$ of Theorem \ref{thm:GP2013}, we know that each $f_h\in\mc{S}^1(\OO,\HH)$ and $\frac{\partial f_h}{\partial x^c}=0$ on $\OO$. Thus, Lemma \ref{lem:real-analytic} ensures that each $f_h$ belongs to $\mscr{C}^\omega(\OO,\HH)$. As an immediate consequence, $f$ belongs to $\mscr{C}^\omega(\OO,\HH)$ as well.
\end{proof}


\subsection{Proof of Theorem \ref{thm:GP2013-sbs}}

\begin{lemma}\label{lem:sbs}
Let $f:\OO\to\HH$ be a function such that
\begin{equation}\label{eq:q}
\text{$f\in\mscr{C}^0_{\mr{sl}}(\OO,\HH)$, $f_*\in\mscr{C}^1_\sbs(\OO_*,\HH)$ and $\debarI f_I=0$ on $\OO_I\setminus\R$ for all $I\in\sq_\HH$.}
\end{equation}
Then $f$ is a slice regular function in the sense of \cite[Definition 8]{AIM2011}, i.e., $f\in\mc{S}^1(\OO,\HH)$ and $\frac{\partial f}{\partial x^c}=0$ on $\OO$.
\end{lemma}
\begin{proof}
We adapt to the present situation the proof of implication $(\mr{ii})\Rightarrow(\mr{i})$ of \cite[Theorem 2.4]{global2013}, which is a generalized version of the Gentili-Struppa `independence from $I$' argument. 

Up to a translation of $\OO$ in $\HH$ along a suitable real number, we can assume that $0\in\OO$. Let $\rho>0$ be such that the open ball $B_\rho$ of $\HH$ centered at $0$ with radius $\rho$ is contained in $\OO$. Since $\OO$ is the circularization of $D\subset\C$ in $\HH$, the open ball $\mathsf{B}_\rho$ of $\C$ centered at $0$ of radius $\rho$ is contained in $D$. For each $I\in\sq_\HH$, we define $B_{\rho,I}:=B_\rho\cap\C_I$.

Let $I,J\in\sq_\HH$ be such that $I\perp J$, and let $g_1,g_2:\OO_I\to\C_I$ be the unique functions such that $f_I=g_1+g_2J$. Thanks to \eqref{eq:q}, $g_1$ and $g_2$ are continuous, their restrictions to $\OO_I\setminus\R$ are of class~$\mscr{C}^1$ and
\[
0=\debarI f_I=\debarI g_1+(\debarI g_2)J \quad \text{on $\OO_I\setminus\R$.}
\]
Thus, the functions $g_1,g_2:\OO_I\to\C_I$ are continuous on the whole $\OO_I$ and holomorphic on $\OO_I\setminus\R$ w.r.t. the complex structure on $\C_I$ induced by the (left) multiplication by $I$. Thanks to Morera's theorem, $g_1$ and $g_2$ turns out to be  holomorphic on the whole $\OO_I$; in particular, $g_1$ and $g_2$ are real analytic. For each $l\in\{1,2\}$, we can expand $g_l$ as follows: $g_l(z_I)=\sum_{h\in\N}z_I^h\frac{1}{h!}\frac{\partial^hg_l}{\partial\alpha^h}(0)$ for all $z_I=\alpha+I\beta\in B_{p,I}$ with $\alpha,\beta\in\R$, where $z_I^h:=(z_I)^h$. Consequently, for each $I\in\sq_\HH$, $f_I=g_1+g_2J$ is real analytic and
\[
f_I(z_I)=g_1(z_I)+g_2(z_I)J=\sum_{h\in\N}z_I^h\frac{1}{h!}\left(\frac{\partial^hg_1}{\partial\alpha^h}(0)+\frac{\partial^hg_2}{\partial\alpha^h}(0)J\right)=\sum_{h\in\N}z_I^h\frac{1}{h!}\frac{\partial^hf}{\partial\alpha^h}(0)
\]
for all $z_I\in B_{p,I}$. Set $a_h:=\frac{1}{h!}\frac{\partial^hf}{\partial\alpha^h}(0)$ for all $n\in\N$. Note that the $a_h$'s do not depend on $I$.

Given $K,H\in\sq_\HH$, we define the real analytic function $g_{K,H}:D\to\HH$ by
\[
g_{K,H}(z):=f_K(z_K)-\frac{1}{2}\big(f_H(z_H)+f_H(\overline{z_H})\big)+K\frac{H}{2}\big(f_H(z_H)-f_H(\overline{z_H})\big)
\]
for all $z=\alpha+i\beta\in\mathsf{B}_\rho$, where $z_K=\alpha+K\beta$ and $z_H=\alpha+H\beta$. For each $h\in\N$, let $p_h,q_h$ be the polynomials in $\R[\alpha,\beta]$ such that $z^h=p_h(\alpha,\beta)+iq_h(\alpha,\beta)$ for all $z=\alpha+i\beta\in\C$. Note that, for all $z=\alpha+i\beta\in\mathsf{B}_\rho$, it holds:
\begin{align*}
z_H^h+\overline{z_H^h}=2p_h(\alpha,\beta), \qquad z_H^h-\overline{z_H^h}=2Hq_h(\alpha,\beta)
\end{align*}
and hence
\begin{align*}
g_{K,H}(z)&=\sum_{h\in\N}\left(z_K^ha_h-\frac{1}{2}\big(z_H^ha_h+\overline{z_H}^ha_h\big)+K\frac{H}{2}\big(z_H^ha_h-\overline{z_H}^ha_h\big)\right)=\\
&=\sum_{h\in\N}\left(z_K^h-\frac{1}{2}\big(z_H^h+\overline{z_H^h}\big)+K\frac{H}{2}\big(z_H^h-\overline{z_H^h}\big)\right)a_h=\\
&=\sum_{h\in\N}\left(p_h(\alpha,\beta)+Kq_h(\alpha,\beta)-p_h(\alpha,\beta)+K\big(H(Hq_h(\alpha,\beta))\big)\right)a_h=0.
\end{align*}
Since $D$ is connected and $g_{K,H}:D\to\HH$ is a real analytic function vanishing on $\mathsf{B}_\rho\subset D$, the principle of analytic continuation implies that $g_{K,H}$ vanishes on the whole $D$. This proves that $f:\OO\to\HH$ satisfies the representation formula, i.e.,
\begin{equation}\label{eq:repres-formula}
f_K(z_K)=\frac{1}{2}\big(f_H(z_H)+f_H(\overline{z_H})\big)-K\frac{H}{2}\big(f_H(z_H)-f_H(\overline{z_H})\big)
\end{equation}
for all $K,H\in\sq_\HH$ and for all $z\in D$. Fix $H\in\sq_\HH$ and define $F_1,F_2:D\to\HH$ by
\[
F_1(z):=\frac{1}{2}\big(f_H(z_H)+f_H(\overline{z_H})\big)
\quad\text{ and }\quad
F_2(z):=-\frac{H}{2}\big(f_H(z_H)-f_H(\overline{z_H})\big)
\]
for all $z=\alpha+i\beta\in D$. Since $F_1$ is even and $F_2$ is odd w.r.t. $\beta$, the function $F:=F_1+iF_2:D\to\HH\otimes_\R\C$ is a stem function. Moreover, it is of class $\mscr{C}^1$ ($\mscr{C}^\omega$ indeed). Thanks to \eqref{eq:repres-formula}, $f$ coincides with $\I(F)$; in particular, $f\in\mc{S}^1(\OO,\HH)$. Since $\debarI f_I=0$ on the whole $\OO_I$ for all $I\in\sq_\HH$, Lemma \ref{lem:de-n} with $n=1$ (or \cite[Proposition 8]{AIM2011}) ensures that $\frac{\partial f}{\partial x^c}=0$ on $\OO$. 
\end{proof}

\begin{corollary}\label{cor:sbs}
Let $f:\OO\to\HH$ be a function. The following assertion hold.
\begin{itemize}
 \item[$(\mr{i})$] If $f\in\mscr{C}^1_\sbs(\OO,\HH)$ and $\debarI f_I=0$ on $\OO_I$ for all $I\in\sq_\HH$, then $f\in\mc{SR}(\OO,\HH)$.
 \item[$(\mr{ii})$] If $f\in\mscr{C}^0_{\mr{sl}}(\OO,\HH)$, $f_*\in\mscr{C}^1(\OO_*,\HH)$ and $f_*\in\ker(\thetabar)$, then $f\in\mc{SR}(\OO,\HH)$.
\end{itemize}
\end{corollary}
\begin{proof}
$(\mr{i})$ Note that, if $f\in\mscr{C}^1_\sbs(\OO,\HH)$, then $f$ belongs to $\mscr{C}^0_\sbs(\OO,\HH)$ and $f_*\in\mscr{C}^1_\sbs(\OO_*,\HH)$. Thus, Lemma \ref{lem:sbs} implies that $f\in\mc{S}^1(\OO,\HH)$ and $\frac{\partial f}{\partial x^c}=0$ on $\OO$. Thanks to the implication $(\mr{e})\Rightarrow(\mr{a})$ of Theorem \ref{thm:GP2013}, we have that $f\in\mc{SR}(\OO,\HH)$.

$(\mr{ii})$ Suppose that $f\in\mscr{C}^0_{\mr{sl}}(\OO,\HH)$, $f_*\in\mscr{C}^1(\OO_*,\HH)$ and $f_*\in\ker(\thetabar)$; evidently,  $f_*\in\mscr{C}^1_\sbs(\OO_*,\HH)$ as well. By Lemma \ref{lem:1} with $n=1$, $\debarI f_I=0$ on $\OO_I\setminus\R$. As in the proof of the preceding point, Lemma \ref{lem:sbs} and implication $(\mr{e})\Rightarrow(\mr{a})$ of Theorem \ref{thm:GP2013} imply that $f\in\mc{SR}(\OO,\HH)$. 
\end{proof}

\begin{proof}[Proof of Theorem \ref{thm:GP2013-sbs}]
Implication $(\mr{a})\Rightarrow(\mr{a}')$ is an immediate consequence of the definition of $\mc{SR}(\OO,\HH)$, because $\mscr{C}^1(\OO,\HH)\subset\mscr{C}^1_\sbs(\OO,\HH)$. Implication $(\mr{a})\Rightarrow(\mr{b}')$ follows immediately from implication $(\mr{a})\Rightarrow(\mr{b})$ of Theorem \ref{thm:GP2013}, because $\mscr{C}^0(\OO,\HH)\subset\mscr{C}^0_\sbs(\OO,\HH)$. Implications $(\mr{a}')\Rightarrow(\mr{a})$ and $(\mr{b}')\Rightarrow(\mr{a})$ coincide with points $(\mr{i})$ and $(\mr{ii})$ of Corollary \ref{cor:sbs}, respectively.
\end{proof}


\subsection{Proofs of Theorem \ref{main} and Corollary \ref{cor:main}}

\begin{lemma}\label{lem:leibniz}
The following assertions hold.
\begin{itemize}
 \item[$(\mr{i})$] Let $f,g\in\mscr{C}^1_\sbs(\OO,\HH)$, let $fg\in\mscr{C}^1_\sbs(\OO,\HH)$ be their pointwise product and let $I\in\sq_\HH$. Suppose that $f(\OO_I)\subset\C_I$. Then $\debarI(fg)_I=(\debarI f_I)g_I+f_I(\debarI g_I)$. 
 \item[$(\mr{ii})$] Let $h\in\N^*$, let $g\in\mscr{C}^1_\sbs(\OO,\HH)$ and let $\overline{x}^hg\in\mscr{C}^1_\sbs(\OO,\HH)$ be the function $x\mapsto\overline{x}^hg(x)$. Choose $I\in\sq_\HH$ arbitrarily. Then $\debarI(\overline{x}^hg)_I=h\overline{x_I}^{h-1}g_I+\overline{x_I}^h\,\debarI g_I$, where $x_I:\OO_I\hookrightarrow\HH$ is the inclusion map and $\overline{x_I}^0:=1$.
  \item[$(\mr{iii})$] Let $h\in\N^*$, let $g\in\mscr{C}^1(\OO_*,\HH)$ and let $\overline{x}_*^h g\in\mscr{C}^1(\OO_*,\HH)$ be the function $x\mapsto\overline{x}^h g(x)$. Then $\thetabar(\overline{x}_*^h g)=h\overline{x}_*^{h-1}g+\overline{x}_*^h\,\thetabar g$ on $\OO_*$, where $\overline{x}_*^0:=1$.
 \item[$(\mr{iv})$] Let $h\in\N^*$ and let $g\in\mc{S}^1(\OO,\HH)$. Then $\overline{x}^h g\in\mc{S}^1(\OO,\HH)$ and $\frac{\partial({x}^h g)}{\partial x^c}=h\overline{x}^{h-1}g+\overline{x}^h\frac{\partial g}{\partial x^c}$ on $\OO$, where $\overline{x}^0:=1$.
\end{itemize} 
\end{lemma}
\begin{proof}
$(\mr{i})$ Define $\widehat{f}_I,\widehat{g}_I:D\to\HH$ by $\widehat{f}_I(z):=f(x_I)$ and $\widehat{g}_I(z):=g(x_I)$ for all $z=\alpha+i\beta\in D$ with $\alpha,\beta\in\R$ and $x_I:=\alpha+I\beta$. Since $f(\OO_I)\subset\C_I$, it holds $I\widehat{f}_I=\widehat{f}_II$ on $D$. Thus,
\begin{align*}
2\debarI (fg)_I(x_I)&=2\debarI (f_Ig_I)(x_I)=\frac{\partial(\widehat{f}_I\widehat{g}_I)}{\partial\alpha}(z)+I\frac{\partial(\widehat{f}_I\widehat{g}_I)}{\partial\beta}(z)=\\
&=\frac{\partial \widehat{f}_I}{\partial\alpha}(z)\widehat{g}_I(z)+\widehat{f}_I(z)\frac{\partial \widehat{g}_I}{\partial\alpha}(z)+I\frac{\partial \widehat{f}_I}{\partial\beta}(z)\widehat{g}_I(z)+I\widehat{f}_I(z)\frac{\partial\widehat{g}_I}{\partial\beta}(z)=\\
&=2(\debarI f_I(x_I))\widehat{g}_I(z)+\widehat{f}_I(z)\frac{\partial \widehat{g}_I}{\partial\alpha}(z)+\widehat{f}_I(z)I\frac{\partial\widehat{g}_I}{\partial\beta}(z)=\\
&=2(\debarI f_I(x_I))\widehat{g}_I(z)+2\widehat{f}_I(z)(\debarI g_I(x_I))=2(\debarI f_I(x_I))g_I(x_I)+2f_I(x_I)(\debarI g_I(x_I)).
\end{align*}

$(\mr{ii})$ Since $\overline{x_I}^h\in\C_I$ for all $x_I\in\OO_I$ and for all $h\in\N^*$, by the preceding point, it suffices to show that $\debarI(\overline{x_I}^h)=h\overline{x}_I^{h-1}$. Let us proceed by induction on $h\geq1$. First, consider the case $h=1$. We have: $\debarI \overline{x_I}=\frac{1}{2}\big(\frac{\partial (\alpha-I\beta)}{\partial\alpha}+I\frac{\partial (\alpha-I\beta)}{\partial\beta}\big)=\frac{1}{2}(1-I^2)=1$, as desired. Let $h\geq2$. Using the induction hypothesis and the preceding point again, we obtain: $\debarI(\overline{x_I}^h)=\debarI(\overline{x_I}\,\overline{x_I}^{h-1})=(\debarI\overline{x_I})\overline{x_I}^{h-1}+\overline{x_I}\debarI\overline{x_I}^{h-1}=\overline{x_I}^{h-1}+\overline{x_I}(h-1)\overline{x_I}^{h-2}=h\overline{x_I}^{h-1}$.

$(\mr{iii})$ This point follows at once from Lemma \ref{lem:1} with $n=1$ and the preceding point applied slice-by-slice.

$(\mr{iv})$ If $g=\I(G)$, then $\overline{x}^h g=\I(\overline{z}^h G)$, where $\overline{z}^h G:D\to\HH\otimes_\R\C$ is the stem function obtained via the pointwise product of $z\mapsto\overline{z}^h$ and $z\mapsto G(z)$, see \cite[Section 5, especially Remark 7]{AIM2011}. Now formula $\frac{\partial({x}^h g)}{\partial x^c}=h\overline{x}^{h-1}g+\overline{x}^h\frac{\partial g}{\partial x^c}$ follows at once from Lemma \ref{lem:de-n} with $n=1$ and the preceding point $(\mr{ii})$ applied slice-by-slice.
\end{proof}

\begin{proof}[Proof of Theorem \ref{main}]

$(\mr{c})\Leftrightarrow(\mr{c}')\Leftrightarrow(\mr{c}'')$ These equivalences follow immediately from Lemmas \ref{lem:1} and \ref{lem:de-n} and the density of $\OO_I\setminus\R$ in $\OO_I$.  

$(\mr{a})\Rightarrow(\mr{c})\,\&\,(\mr{b})$ Let $f\in\mc{SP}^n(\OO,\HH)$. Write $f$ as follows: $f=\sum_{h=0}^{n-1}\overline{x}^hf_h$ on $\OO$, where $f_0,\ldots,f_{n-1}\in\mc{SR}(\OO,\HH)$. By implication $(\mr{a})\Rightarrow(\mr{e})$ of Theorem \ref{thm:GP2013}, we know that each $f_h$ belongs to $\mc{S}^1(\OO,\HH)$ and, if $F^{(h)}:D\to\HH\otimes_\R\C$ is the stem function inducing $f_h$, then $F^{(h)}$ is holomorphic. In particular, $\sum_{h=0}^{n-1}\overline{z}^hF^{(h)}:D\to\HH\otimes_\R\C$ is a stem function of class $\mscr{C}^n$ ($\mscr{C}^\omega$ indeed), whose induced slice function coincides with $f$. Thus, $f$ belongs to $\mc{S}^n(\OO,\HH)$. Let $I\in\sq_\HH$ and let $f_{h,I}:\OO_I\to\HH$ be the restriction of each $f_h$ to $\OO_I$. Set $x_I:=\alpha+I\beta\in\OO_I$ with $\alpha,\beta\in\R$. Bearing in mind that $\debarI f_{h,I}=0$ on $\OO_I$ for all $h\in\{0,\ldots,n-1\}$, an iterated application of Lemma \ref{lem:leibniz}($\mr{ii}$) implies that
\[\textstyle
\debarI^{\,l} f_I=\debarI^{\,l}\big(\sum_{h=0}^{n-1}\overline{x}_I^hf_{h,I}\big)=\sum_{h=l}^{n-1}h(h-1)\cdots(h-l+1)\overline{x}_I^{h-l}f_{h,I} \quad\text{on $\OO_I$}
\]
for all $l\in\{0,\ldots,n-1\}$. In particular, it follows that $\debarI^{\,n-1}f_I=(n-1)!f_{n-1,I}$ and hence $\debarI^{\,n} f_I=(n-1)!\,\debarI f_{n-1,I}=0$ on $\OO_I$. This proves that $\debarIn f_I=0$ on $\OO_I$ for all $I\in\sq_\HH$. 

Now, using implication $(\mr{a})\Rightarrow(\mr{c})$ of Theorem \ref{thm:GP2013} and an iterated application of Lemma \ref{lem:leibniz}$(\mr{iii})$, we obtain that $\thetabar^n f_*=0$ on $\OO_*$.

$(\mr{c}'')\Rightarrow(\mr{a})$ Let us proceed by induction on $n\in\N^*$, following the strategy of \cite[Section 1.1, top of page 11]{balk}. The case $n=1$ coincides with implication $(\mr{e})\Rightarrow(\mr{a})$ of Theorem \ref{thm:GP2013}, because $\mc{SP}^1(\OO,\HH)=\mc{SR}(\OO,\HH)$. Let $n\geq2$. Note that $\frac{\partial f}{\partial x^c}\in\mc{S}^{n-1}(\OO,\HH)$ and $\big(\frac{\partial}{\partial x^c}\big)^{n-1}\big(\frac{\partial f}{\partial x^c}\big)=\big(\frac{\partial}{\partial x^c}\big)^nf=0$ on $\OO$. By induction hypothesis, $\frac{\partial f}{\partial x^c}=\sum_{h=0}^{n-2}\overline{x}^hf_h$ on $\OO$ for some $f_h\in\mc{SR}(\OO,\HH)$.  Combining implication $(\mr{a})\Rightarrow(\mr{e})$ of Theorem \ref{thm:GP2013} with Lemma \ref{lem:leibniz}$(\mr{iv})$, we deduce that $\big(\frac{\partial}{\partial x^c}\big)\big(\sum_{h=0}^{n-2}\frac{1}{h+1}\overline{x}^{h+1}f_h\big)=\sum_{h=0}^{n-2}\overline{x}^hf_h$ on $\OO$. Thus, $\big(\frac{\partial}{\partial x^c}\big)\big(f-\sum_{h=0}^{n-2}\frac{1}{h+1}\overline{x}^{h+1}f_h\big)=0$ on $\OO$. Implication $(\mr{e})\Rightarrow(\mr{a})$ of Theorem \ref{thm:GP2013} ensures that $f-\sum_{h=0}^{n-2}\frac{1}{h+1}\overline{x}^{h+1}f_h\in\mc{SR}(\OO,\HH)$ and hence $f\in\mc{SP}^n(\OO,\HH)$.

$(\mr{b})\Rightarrow(\mr{a})$ Let us proceed by induction on $n\in\N^*$. The case $n=1$ follows immediately from Lemma \ref{lem:reg-n} (with $n=0$) and implication $(\mr{b}')\Rightarrow(\mr{a})$ of Theorem \ref{thm:GP2013-sbs}. Let $n\geq2$. Let $f\in\mc{S}^{n-1}(\OO,\HH)\subset\mc{S}^1(\OO,\HH)$ such that $f_*\in\mscr{C}^n(\OO_*,\HH)$ and $\thetabar^nf_*=0$ on $\OO_*$; in particular, $\frac{\partial f}{\partial x^c}\in\mc{S}^{n-2}(\OO,\HH)$. Let $F:D\to\HH\otimes_\R\C$ be the stem function of class $\mscr{C}^{n-1}$ inducing $f$. Since $f_*\in\mscr{C}^n(\OO_*,\HH)$, formula \eqref{eq:stem} (applied with a fixed $I\in\sq_\HH$) implies that $F|_{D\setminus\R}$ is of class~$\mscr{C}^n$ and hence $\frac{\partial F}{\partial\overline{z}}\big|_{D\setminus\R}$ is of class $\mscr{C}^{n-1}$. Applying Remark \ref{rem:reg-n} to $\frac{\partial F}{\partial\overline{z}}$, we obtain that $\frac{\partial f}{\partial x^c}\big|_{\OO_*}\in\mscr{C}^{n-1}(\OO_*,\HH)$. Applying Lemmas \ref{lem:1} and \ref{lem:de-n} (with $n=1$) slice-by-slice, it follows that $\frac{\partial f}{\partial x^c}\big|_{\OO_*}=\thetabar f_*$ on~$\OO_*$. As a consequence, we have: $\thetabar^{n-1}\big(\frac{\partial f}{\partial x^c}\big|_{\OO_*}\big)=\thetabar^{n-1}(\thetabar f_*)=\thetabar^nf_*=0$ on~$\OO_*$. We have just proved that $\frac{\partial f}{\partial x^c}\in\mc{S}^{n-2}(\OO,\HH)$, $\frac{\partial f}{\partial x^c}\big|_{\OO_*}\in\mscr{C}^{n-1}(\OO_*,\HH)$ and $\frac{\partial f}{\partial x^c}\big|_{\OO_*}\in\ker(\thetabar^{n-1})$. By the induction hypothesis, it turns out that $\frac{\partial f}{\partial x^c}=\sum_{h=0}^{n-2}\overline{x}^hf_h$ on $\OO$ for some $f_0,\ldots,f_{n-2}\in\mc{SR}(\OO,\HH)$. Thus, bearing in mind implication $(\mr{a})\Rightarrow(\mr{e})$ of Theorem \ref{thm:GP2013} and \cite[Section 5, especially Remark 7]{AIM2011}, we have that the function $g:=f-\sum_{h=0}^{n-2}\frac{1}{h+1}\overline{x}^{h+1}f_h$ belongs to $\mc{S}^1(\OO,\HH)$. Moreover, by Lemma~\ref{lem:leibniz}$(\mr{iv})$, it holds $\frac{\partial g}{\partial x^c}=0$ on $\OO$. Implication $(\mr{e})\Rightarrow(\mr{a})$ of Theorem \ref{thm:GP2013} implies that $g\in\mc{SR}(\OO,\HH)$, completing the proof.
\end{proof}

\begin{proof}[Proof of Corollary \ref{cor:main}]
This result follows immediately from equivalence $(\mr{a})\Leftrightarrow(\mr{c}')$ of Theorem \ref{main}, and Lemma \ref{lem:r-a-spn}.
\end{proof}


\vspace{1em}

\noindent {\bf Acknowledgement.} This work was supported by GNSAGA of INdAM.



\end{document}